\documentclass[a4paper,12pt]{amsart}

\usepackage[T1]{fontenc}
\usepackage{inputenc}

\usepackage[english]{babel}
\usepackage{amsmath}
\usepackage{latexsym}
\usepackage{marvosym}
\usepackage{amsfonts}
\usepackage{amsthm}
\usepackage{float}
\usepackage{color}
\usepackage{epsfig}
\usepackage{graphicx}

\usepackage{enumerate}

\usepackage{stmaryrd}

\usepackage{tikz}
\usetikzlibrary{trees}
\usepackage{verbatim}
\usetikzlibrary{shapes,arrows}
\usepackage{graphicx}

\newcommand\N{\mathbb{N}}

\newcommand\E{\mathbb{E}}

\newcommand\R{\mathbb{R}}

\newcommand\T{\mathbb{T}}
\newcommand\G{\mathbb{G}}

\newcommand{\veps}{\varepsilon}

\newcommand\calF{\mathcal{F}}
\newcommand\calG{\mathcal{G}}
\newcommand\calB{\mathcal{B}}
\newcommand\calO{\mathcal{O}}
\newcommand\calN{\mathcal{N}}

\newcommand\calV{\mathcal{V}}
\newcommand\calA{\mathcal{A}}
\newcommand\calW{\mathcal{W}}
\newcommand\calR{\mathcal{R}}
\newcommand\calT{\mathcal{T}}
\newcommand\calH{\mathcal{H}}
\newcommand\calK{\mathcal{K}}

\newcommand\wh{\widehat}

\newcommand\Var{\text{Var}}

\newcommand{\as}{\hspace{20pt} \text{a.s.}}

\renewcommand{\(}{\left(}
\renewcommand{\)}{\right)}

\def\build#1_#2^#3{\mathrel{\mathop{\kern 0pt#1}\limits_{#2}^{#3}}}
\def\liml{\build{\longrightarrow}_{}^{{\mbox{$\mathcal L$}}}}

\def\H#1{\textup{\textbf{(H.\ref{#1})}}}

\numberwithin{equation}{section}
\theoremstyle{plain}
\newtheorem{TD}{Theorem-Definition}[section]

\newtheorem{Prop}[TD]{Proposition}
\newtheorem{Theo}[TD]{Theorem}

\newtheorem{Lem}[TD]{Lemma}
\newtheorem{Rem}[TD]{Remark}

\begin{document}

\title[A Rademacher-Menchov approach for RCBAR processes]
{A Rademacher-Menchov approach for random coefficient bifurcating autoregressive processes}
\author{Bernard Bercu}
\author{Vassili Blandin}
\dedicatory{\normalsize Universit\'e Bordeaux 1}
\address{Universit\'e Bordeaux 1, Institut de Math\'ematiques de Bordeaux, UMR CNRS 5251, and INRIA Bordeaux, team ALEA,
	351 cours de la lib\'eration, 33405 Talence cedex, France.}

\email{bernard.bercu@math.u-bordeaux1.fr}
\email{vassili.blandin@math.u-bordeaux1.fr}


\begin{abstract}
We investigate the asymptotic behavior of the least squares estimator of the unknown parameters of random coefficient bifurcating autoregressive processes. Under suitable assumptions on inherited and environmental effects, we establish the almost sure convergence of our estimates. In addition, we also prove a quadratic strong law and central limit theorems. Our approach mainly relies on asymptotic results for vector-valued martingales together with the well-known Rademacher-Menchov theorem.
\end{abstract}

\maketitle


\section{Introduction}


The purpose of this paper is to study random coefficient bifurcating autoregressive processes (RCBAR). One can see those processes in two different ways. The first one is to see them as random coefficient autoregressive processes (RCAR) adapted to binary tree structured data, the second one is to consider those processes as the association of RCAR processes and bifurcating autoregressive processes (BAR). BAR processes have been first studied by Cowan and Staudte \cite{CowanStaudte} when RCAR processes have been first investigated by Nicholls and Quinn \cite{NichollsQuinn1,NichollsQuinn2}. The RCBAR structure allows us to reckon with environmental effects and inherited effects in order to better take into account the evolution of the characteristic under study. One shall see cell division as an example of binary tree structured data.\\

Let us detail what a RCBAR process is. The first individual is designated as the individual 1 and each individual $n$ leads to individuals $2n$ and $2n+1$. $X_n$ will stand for the characteristic under study of individual $n$. We can now make explicit the first-order RCBAR process which is given, for all $n\geq 1$, by
$$\begin{cases} X_{2n}=a_n X_n +\veps_{2n},\\ X_{2n+1} = b_n X_n +\veps_{2n+1}. \end{cases}$$
The driven noise sequence $(\veps_{2n},\veps_{2n+1})$ represents the environmental effect while the random coefficient sequence $(a_n,b_n)$ represents the inherited effect. Keeping in mind the example of cell division, we assume that $\veps_{2n}$ and $\veps_{2n+1}$ are correlated in order to take into account the environmental effect on two sister cells.\\

Our goal is to study the asymptotic behavior of the least squares estimators of the unknown parameters of first-order RCBAR processes. In contrast with the previous work of Blandin \cite{Vassili2} where the asymptotic behavior of weighted least squares estimators were investigated, we propose here to make use of a totally different strategy based on the standard least squares (LS) estimators together with the well-known Rademacher-Menchov theorem. The martingale approach for BAR processes has been first suggested by Bercu et al.~\cite{BercuBDSAGP}, followed by the recent contribution of de Saporta et al.~\cite{BDSAGPMarsalle,BDSAGPMarsalleRCBAR}. We also refer the reader to Blandin \cite{Vassili} for the study of bifurcating integer-valued autoregressive processes. Our approach relies on the Rademacher-Menchov theorem which allows us to study the LS estimates in a different way as in de Saporta et al.~\cite{BDSAGPMarsalleRCBAR}. In particular, we drastically reduce the moment assumptions on the random coefficient sequence $(a_n,b_n)$ and on the driven noise sequence $(\veps_{2n},\veps_{2n+1})$. We shall also make use of the strong law of large numbers for martingales \cite{Duflo} and also the central limit theorem for martingales \cite{Duflo,HallHeyde} in order to study the asymptotic behavior of our LS estimates. The martingale approach of this paper has also been used by Basawa and Zhou \cite{BasawaZhou,ZhouBasawa,ZhouBasawa2}.\\

Since several methods have been proposed for the study of BAR processes, we tried to take into consideration each of them. In this way, we took into account the classical BAR approach as used by Huggins and Basawa \cite{HugginsBasawa99,HugginsBasawa2000} and by Huggins and Staudte \cite{HugginsStaudte} who investigated the evolution of cell diameters and lifetimes. We were also inspired by the bifurcating Markov chain approach brought in by Guyon \cite{Guyon} and applied by Delmas and Marsalle \cite{DelmasMarsalle}. We also reckoned with the analogy with the Galton-Watson processes as in Delmas and Marsalle \cite{DelmasMarsalle} and Heyde and Seneta \cite{HeydeSeneta}. Even though we chose to use LS estimates, different methods have been investigated for  parameter estimation in RCAR processes. While Koul and Schick \cite{KoulSchick} used an M-estimator, Aue et al. \cite{AueHorvathSteinebach} tackled a quasi-maximum likelihood approach. Vanecek \cite{Vanecek} used an estimator first introduced by Schick \cite{Schick}. On their side, Hwag et al. \cite{HwangBasawaKim} studied the critical case where the environmental effect follows a Rademacher distribution.\\

The paper is organised as follows. We will explain more accurately the model we will consider in Section 2, leading to Section 3 where we will give explicitly our LS estimates of the unknown parameters under study. The martingale point of view chosen in this paper will be highlighted in Section 4. All our results about the asymptotic behavior of our LS estimates will be stated in Section 5, in particular the almost sure convergence, the quadratic strong law and the asymptotic normality. Section 6 is devoted to the Rademacher-Menchov theorem. All technical proofs are postponed to the last sections.


\section{Random coefficient bifurcating autoregressive processes}\label{section2}


We will study the first-order RCBAR process given, for all $n\geq1$, by
\begin{equation} \label{systori}
\begin{cases}
X_{2n} &= a_n X_n + \veps _{2n},\\
X_{2n+1} &= b_n X_n + \veps_{2n+1},
\end{cases}
\end{equation}
where $X_1$ is the ancestor of the process and $(\veps_{2n},\veps_{2n+1})$ is the driven noise of the process. We will suppose that $\E[X_1^{16}]<\infty$ and we will also assume that the two sequences $(a_n,b_n)$ and $(\veps_{2n},\veps_{2n+1})$ are independent and indentically distributed and are also mutually independent.
An RCBAR can be seen as a first-order autoregressive process on a binary tree, each node of this tree representing an individual and the first node being the ancestor. For all $n\geq1$, $\G_n$ will stand for the $n$-th generation, that is to say
$$\G_n=\{2^{n},2^{n}+1,\hdots,2^{n+1}-1\}.$$
We will also denote by $\T_n$ the set of all individuals up to the $n$-th generation, namely
$$\T_n = \bigcup_{n=0}^n \G_n.$$
One can immediately see that the cardinality $|\G_n|$ of $\G_n$ is $2^n$, while that of $\T_n$ is $2^{n+1}-1$. We will note $\G_{r_n}$ the generation of individual $n$, $r_n=\log_2(n)$. Let us recall that the two offspring of individual $n$ are individuals $2n$ and $2n+1$ and conversely, the direct ancestor of individual $n$ is individual $[n/2]$ where $[x]$ stands for the integer part of $x$.

\begin{figure}[h!]
\hspace{-2cm}
\includegraphics[scale=0.7]{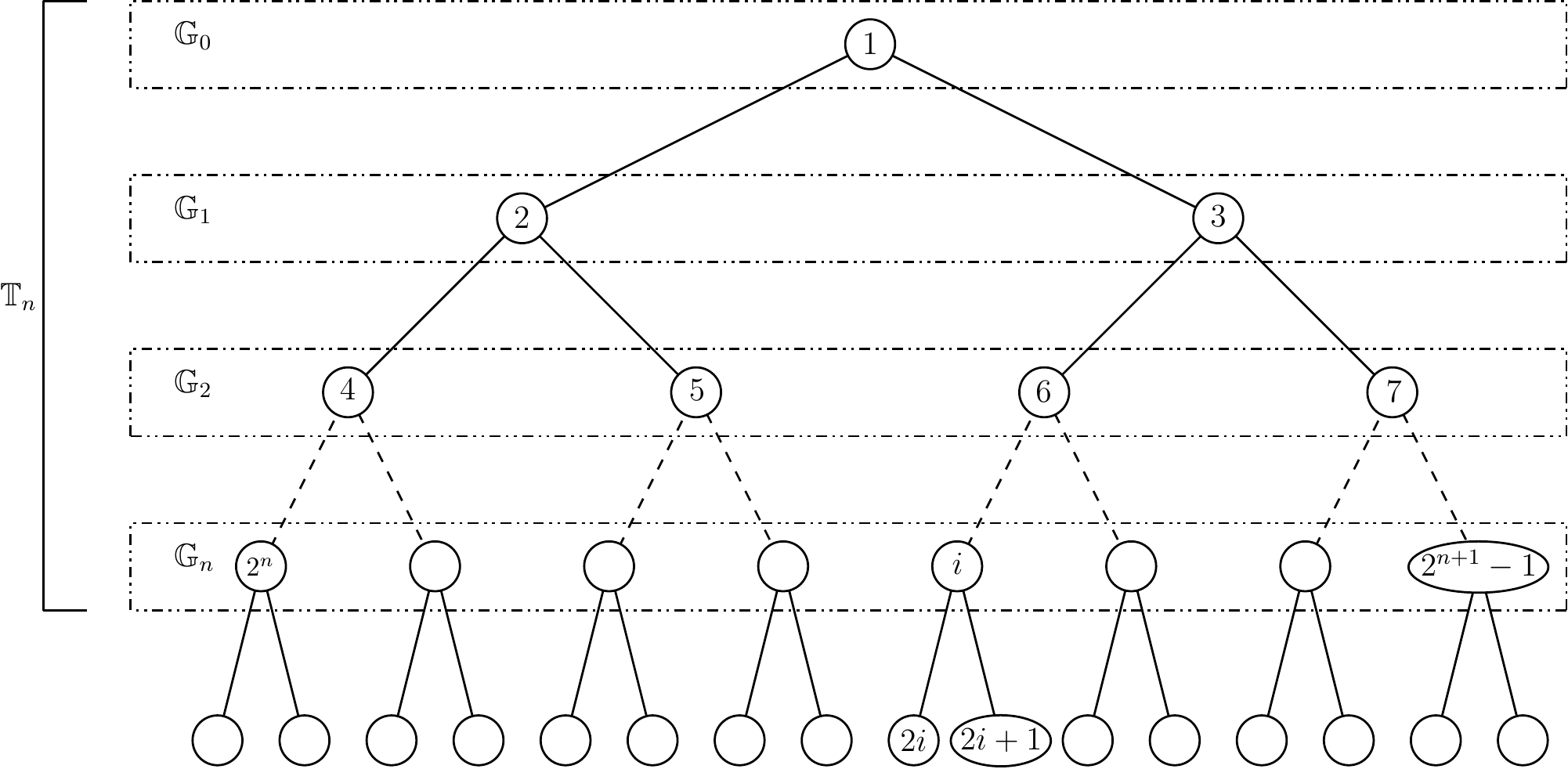}
\caption{The tree associated with the RCBAR}
\end{figure}


\section{Least squares estimators}\label{LSE}


Let $(\calF_n)$ be the natural filtration associated with the generations of our first-order RCBAR $(X_n)$, namely $\calF_n=\sigma\{X_k,k\in\T_n\}$ for all $n\in\N$. In all the sequel, we will assume that for all $n\geq0$ and for all $k\in\G_n$,
\begin{equation}\label{hypesp}
\begin{array}{c}
\E[a_k|\calF_n]=a, \hspace{20pt} \E[b_k|\calF_n]=b,\\
\E[\veps_{2k}|\calF_n]=c, \hspace{20pt} \E[\veps_{2k+1}|\calF_n]=d \as
\end{array}
\end{equation}
Consequently, \eqref{systori} can be rewritten as
\begin{equation} \label{systbis}
\begin{cases}
X_{2n} &= aX_n + c + V_{2n},\\
X_{2n+1} &= bX_n + d + V_{2n+1},
\end{cases}
\end{equation}
where, for all $k\in\G_n$, $V_{2k}=X_{2k}-E[X_{2k}|\calF_n]$ and $V_{2k+1}=X_{2k+1}-E[X_{2k}|\calF_n]$. We can rewrite the system \eqref{systbis} in a classic autoregressive form
\begin{equation} \label{formeBAR}
\chi_n = \theta^t \Phi_n + W_n
\end{equation}
where
$$\chi_n=\begin{pmatrix}X_{2n} \\ X_{2n+1}\end{pmatrix}, \hspace{20pt} \Phi_n = \begin{pmatrix}X_{n} \\ 1\end{pmatrix}, \hspace{20pt} \begin{pmatrix}V_{2n} \\ V_{2n+1} \end{pmatrix},$$
and the matrix parameter $\theta$ given by 
$$\theta = \begin{pmatrix}a & b \\ c & d \\ \end{pmatrix}.$$
One of our goal is to estimate $\theta$ from the observation of the $n$ first generation, namely $\T_n$. We will use the least squares estimator $\wh \theta_n$ of $\theta$ which minimizes
$$\Delta_n(\theta) = \sum_{k\in\T_{n-1}} \|\chi_k-\theta^t \Phi_k\|^2.$$
Hence, we obviously have
$$\wh\theta_n=S_{n-1}^{-1} \sum_{k\in\T_{n-1}} \Phi_k\chi_k^t,$$
where $S_n$ is the matrix given by
$$S_n = \sum_{k\in\T_n} \Phi_k\Phi_k^t.$$
In order to avoid any invertibility problem, we will assume that $S_1$ is invertible. Otherwise, we only have to add the identity matrix of order 2, $I_2$, to $S_n$. Moreover, we will make a slight abuse of notation by identifying $\theta$ and $\wh\theta_n$ to
$$\text{vec}(\theta) = \begin{pmatrix} a\\c\\b\\d \end{pmatrix} \hspace{20pt} \text{and} \hspace{20pt} \text{vec}(\wh\theta_n) = \begin{pmatrix} \wh a_n\\\wh c_n\\\wh b_n\\\wh d_n \end{pmatrix}.$$
In this vectorial form, we have
$$\wh\theta_n= \Sigma_{n-1}^{-1}\sum_{k\in\T_n} \begin{pmatrix} X_kX_{2k} \\ X_{2k} \\ X_kX_{2k+1} \\ X_{2k+1} \end{pmatrix},$$
where $\Sigma_n = I_2 \otimes S_n$ and $\otimes$ stands for the standard Kronecker product. Hence, \eqref{formeBAR} yields to
\begin{equation} \label{difftheta}
\wh\theta_n - \theta = \Sigma_{n-1}^{-1}\sum_{k\in\T_n} \begin{pmatrix} X_kV_{2k} \\ V_{2k} \\ X_kV_{2k+1} \\ V_{2k+1} \end{pmatrix}.
\end{equation}
In all this paper, we will make use of the following hypotheses on the moments of the random coefficient sequence $(a_n,b_n)$ and on the driven noise sequence $(\veps_{2n},\veps_{2n+1})$. One can observe that for all $n\geq0$ and for all $k\in\G_n$, the random coefficients $a_k$, $b_k$ and the driven noise $\veps_{2n}$, $\veps_{2n+1}$ are $\calF_{n+1}$-measurable.
\begin{enumerate}[\bf{({H}.}1)]
\item\label{H1} For all $n\geq1$,
\begin{align*}
\E[a_n^{16}] < 1 \hspace{20pt}  &\text{and} \hspace{20pt} \E[b_n^{16}] < 1,\\
\sup_{n\geq1}\E[\veps_{2n}^{16}]<\infty \hspace{20pt}  &\text{and} \hspace{20pt} \sup_{n\geq1}\E[\veps_{2n+1}^{16}]< \infty.
\end{align*}
 \item \label{H2} For all $n\geq0$ and for all $k\in \G_n$
$$\begin{array}{ccccccc}
  \Var[a_k|\calF_n] = \sigma_a^2 \geq 0 & \text{and} &  \Var[b_k|\calF_n] = \sigma_b^2 \geq 0 & & \text{a.s.}
\end{array}$$
$$\begin{array}{ccccccccccc}
  \Var[\veps_{2k}|\calF_n] = \sigma_c^2 > 0 & \text{and} &  \Var[\veps_{2k+1}|\calF_n] = \sigma_d^2 > 0 & & \text{a.s.}
\end{array}$$
 \item \label{H3} It exists $\rho_{ab}^2\leq\sigma_a^2\sigma_b^2$ and $\rho_{cd}^2<\sigma_c^2\sigma_d^2$ such that for all $n\geq0$ and for all $k\in\G_n$ 
$$\E[(a_k-a)(b_k-b)|\calF_n]=\rho_{ab} \as$$
$$\E[(\veps_{2k}-c)(\veps_{2k+1}-d)|\calF_n] = \rho_{cd} \hspace{20pt} \text{ a.s.}$$
Moreover, for all $n\geq0$ and for all $k,l\in\G_n$, with $k\neq l$, $(\veps_{2k},\veps_{2k+1})$ and  $(\veps_{2l},\veps_{2l+1})$ as well as $(a_k,b_k)$ and $(a_l,b_l)$ are conditionally independent given $\calF_n$.
\item \label{H4} One can find $\mu_a^4\geq\sigma_a^4$, $\mu_b^4\geq\sigma_b^4$, $\mu_c^4>\sigma_c^4$ and $\mu_d^4>\sigma_d^4$ such that, for all $n\geq0$ and for all $k\in \G_n$
$$\begin{array}{ccccccc}
  \E\left[\left(a_{k}-a\right)^4|\calF_n\right] = \mu_a^4 & & \text{and} & & \E\left[\left(b_k-b\right)^4|\calF_n\right] = \mu_b^4 & & \text{a.s.}
\end{array}$$
$$\begin{array}{ccccccc}
  \E\left[\left(\veps_{2k}-c\right)^4|\calF_n\right] = \mu_c^4 & & \text{and} & & \E\left[\left(\veps_{2k+1}-d\right)^4|\calF_n\right] = \mu_d^4 & & \text{a.s.}
\end{array}$$
$$\E[\veps_{2k}^4]>\E[\veps_{2k}^2]^2 \hspace{20pt} \text{ and } \hspace{20pt} \E[\veps_{2k+1}^4]>\E[\veps_{2k+1}^2]^2.$$
In addition, it exists $\nu_{ab}^2 \geq \rho_{ab}^2$ and $\nu_{cd}^2 > \rho_{cd}^2$ such that, for all $k\in\G_n$
$$\E[(a_k-a)^2(b_k-b)^2|\calF_n] = \nu_{ab}^2 \hspace{20pt} \text{and} \hspace{20pt} \E[(\veps_{2k}-c)^2(\veps_{2k+1}-d)^2|\calF_n] = \nu_{cd}^2 \hspace{20pt} \text{ a.s.}$$
\item \label {H5} It exists some $\alpha>4$ such that
$$\sup_{n\geq0}\sup_{k\in\G_n} \E[|a_k-a|^\alpha|\calF_n]<\infty, \hspace{20pt} \sup_{n\geq0}\sup_{k\in\G_n} \E[|b_k-b|^\alpha|\calF_n]<\infty \as$$
$$\sup_{n\geq0}\sup_{k\in\G_n} \E[|\veps_{2k}-c|^\alpha|\calF_n]<\infty, \hspace{20pt} \sup_{n\geq0}\sup_{k\in\G_n} \E[|\veps_{2k+1}-d|^\alpha|\calF_n]<\infty \as$$
\end{enumerate}
One can observe that hypothesis \H{H2} allows us to consider a classical BAR process where $a_k=a$ and $b_k=b$ a.s. Moreover, under assumptions \H{H2} and \H{H3}, we have for all $n\geq0$ and for all $k\in\G_n$
\begin{align}
\E[V_{2k}^2|\calF_n] &= \sigma_a^2 X_k^2 +\sigma_c^2 , \hspace{20pt} \E[V_{2k+1}^2|\calF_n] = \sigma_b^2 X_k^2 + \sigma_d^2 \as \label{sigma}\\
&\E[V_{2k}V_{2k+1}|\calF_n] = \rho_{ab} X_k^2 + \rho_{cd} \as \label{rho}
\end{align}
We deduce from \eqref{sigma} that, for all $n\geq1$, we can rewrite $V_{2n}^2$ as
$$V_{2n}^2 = \eta^t\psi_n+v_{2n},$$
where
$$\eta=\begin{pmatrix} \sigma_a^2 \\ \sigma_c^2 \end{pmatrix} \hspace{20pt} \text{and} \hspace{20pt} \psi_n = \begin{pmatrix} X_n^2 \\ 1 \end{pmatrix}.$$
It leads us to estimate the vector of variances $\eta$ by the least squares estimator
\begin{equation}\label{esteta}
\wh\eta_n = Q_{n-1}^{-1} \sum_{k\in\T_{n-1}}\wh V_{2k}^2\psi_k,
\end{equation}
where
$$Q_n = \sum_{k\in\T_n} \psi_k \psi_k^t$$
and for all $k\in\G_n$,
\begin{equation*}
\begin{cases}
\wh V_{2k} &= X_{2k} - \wh a_n X_k - \wh c_n,\\
\wh V_{2k+1} &= X_{2k+1} - \wh b_n X_k - \wh d_n.
\end{cases}
\end{equation*}
We clearly have a similar expression for the estimator of the vector of variances $\zeta = \begin{pmatrix}\sigma_b^2 & \sigma_d^2\end{pmatrix}^t$ by replacing $\wh V_{2k}$ by $\wh V_{2k+1}$ into \eqref{esteta}. By the same token, it follows from \eqref{rho} that, for all $n\geq1$, we can rewrite $V_{2n}V_{2n+1}$ as
$$V_{2n}V_{2n+1}  = \nu^t\psi_n + w_{2n}$$
where $\nu$ is the vector of covariances $\nu = \begin{pmatrix} \rho_{ab} & \rho_{cd} \end{pmatrix}^t$. Therefore, we can estimate $\nu$ by
\begin{equation} \label{estnu}
\wh \nu_n = Q_{n-1}^{-1} \sum_{k\in\T_{n-1}} \wh V_{2k} \wh V_{2k+1} \psi_k.
\end{equation}


\section{A martingale approach}


We already saw that relation \eqref{difftheta} can be rewritten as
\begin{equation}\label{diffthetaM}
\wh \theta_n - \theta = \Sigma_{n-1}^{-1} M_n,
\end{equation}
where
$$M_n = \sum_{k\in\T_n} \begin{pmatrix} X_kV_{2k} \\ V_{2k} \\ X_kV_{2k+1} \\ V_{2k+1}\end{pmatrix}.$$
The key point of this study is to remark that $(M_n)$ is a locally square integrable martingale, which allows us to make use of asymptotic results for martingales. This justifies our vectorial notation introduced previously since most of those asymptotic results have been established for vector-valued martingales. In order to study this martingale, let us rewrite $M_n$ in a more convenient way. Let $\Psi_n = I_2 \otimes \varphi_n$ where $\varphi_n$ is the $2\times2^n$ matrix given by
\begin{equation*}\varphi_n = \begin{pmatrix} X_{2^n} & X_{2^n+1} & \hdots & X_{2^{n+1}-1}  \\
	    1 & 1 & \hdots & 1
\end{pmatrix}.\end{equation*}
The first line of $\varphi_n$ gathers the individuals of the $n$-th generation, $\varphi_n$ can also be seen as the collection of all $\Phi_k$, $k\in\G_n$. Let $\xi_n$ be the random vector of dimension $2^n$ gathering the noise variables of $\G_n$, namely
$$\xi_n^t = \begin{pmatrix} V_{2^n} & V_{2^n+2} & \hdots & V_{2^{n+1}-2} & V_{2^n+1} & V_{2^n+3} & \hdots & V_{2^{n+1}-1}
	    \end{pmatrix}.
$$
The special ordering separating odd and even indices has been made in Bercu et al. \cite{BercuBDSAGP} in order to rewrite $M_n$ as
$$M_n = \sum_{k=1}^n \Psi_k\xi_k.$$
It clearly follows from \H{H1} to \H{H3} that $(M_n)$ is a locally square integrable martingale with increasing process given, for all $n\geq 1$, by

\begin{align}
\langle M\rangle_n &= \sum_{k=0}^{n-1} \Psi_k \E[\xi_{k+1}\xi_{k+1}^t|\calF_k]\Psi_k^t =\sum_{k=0}^{n-1} L_k \hspace{20pt} \text{a.s.} \label{defcrochetM}
\end{align}

\noindent where 
\begin{equation}\label{defLk}
    L_n = \sum_{k\in\G_{n}}\begin{pmatrix} P(X_k) & Q(X_k) \\ Q(X_k) & R(X_k) \end{pmatrix}\otimes\begin{pmatrix} X_k^2 & X_k \\ X_k & 1\end{pmatrix}
\end{equation}
with
\begin{equation}\label{defPQR}
\begin{cases} P(X) = \sigma_a^2X^2 + \sigma_c^2,\\
Q(X) = \rho_{ab} X^2 + \rho_{cd},\\
R(X) = \sigma_b^2 X^2 + \sigma_d^2.
\end{cases}
\end{equation}
The first step of our approach will be to establish the convergence of $\langle M\rangle_n$ properly normalized, from which we will be able to deduce several asymptotic results for our RCBAR estimates.


\section{Main results}


\begin{Lem} \label{CVsomme}
Assume that \H{H1} is statisfied. Then we have
\begin{equation}\label{Lemfonda}
\lim_{n\to\infty} \frac1{|\T_n|} \sum_{k\in\T_n} X_k^p = s_p \as
\end{equation}
where $s_p$ is a constant depending only on the moments of $a_1$, $b_1$, $\veps_{2}$ and $\veps_{3}$ up to the $p$-th order.
\end{Lem}

\begin{Rem}
In particular, we have
$$s_1 = \frac{c+d}{2-(a+b)},$$
$$s_2 = \frac{2}{2-(\sigma_a^2 +\sigma_b^2+a^2+b^2)}\left(\frac{(ac+bd)(c+d)}{2-(a+b)} + \frac{\sigma_c^2+\sigma_d^2+c^2+d^2}{2}\right),$$
and explicit expressions for $s_3$ to $s_{8}$ are given at the end of Section \ref{preuvepremlem}.
\end{Rem}

\begin{Prop} \label{cvcrochet}
 Assume that \H{H1} to \H{H3} are satisfied. Then, we have 
\begin{equation} \label{limcrochet}
 \lim_{n\to\infty} \frac{\langle M\rangle_n}{|\T_{n-1}|} = L \hspace{20pt} \text{ a.s.}
\end{equation}
\noindent where $L$ is the positive definite matrix given by
\begin{equation*} L = 
\begin{pmatrix} \sigma_c^2 & \rho_{cd} \\ \rho_{cd} & \sigma_d^2 \end{pmatrix} \otimes C + \begin{pmatrix} \sigma_a^2 & \rho_{ab} \\ \rho_{ab} & \sigma_b^2 \end{pmatrix} \otimes D,
\end{equation*}
where
\begin{equation}\label{defCD}
C = \begin{pmatrix} s_2 & s_1 \\ s_1 & 1 \end{pmatrix} \hspace{20pt} \text{ and } \hspace{20pt} D = \begin{pmatrix} s_4 & s_3 \\ s_3 & s_2 \end{pmatrix}.
\end{equation}

\end{Prop}

\begin{Rem}
One can observe that we only need to assume for convergence \eqref{limcrochet} that 
$$\E[a_n^{8}]<1, \hspace{20pt} \E[b_n^{8}]<1, \hspace{20pt} \sup_{n\geq1}\E[\veps_{2n}^{8}]<\infty, \hspace{20pt} \sup_{n\geq1}\E[\veps_{2n+1}^{8}]<\infty.$$
\end{Rem}

\noindent Our first result deals with the almost sure convergence of the LS estimator $\wh\theta_n$.

\begin{Theo} \label{CVpstheta}
 Assume that \H{H1} to \H{H3} are satisfied. Then, $\wh\theta_n$ converges almost surely 
to $\theta$ with the almost sure rate of convergence
\begin{equation*}
\|\wh\theta_n-\theta\|^2 =  \calO\left(\frac{n}{|\T_{n-1}|}\right) \hspace{20pt} \text{ a.s.} \label{rate}
\end{equation*}
\noindent In addition, we also have the quadratic strong law
\begin{equation}\label{quadratic1}
 \lim_{n\to\infty} \frac1n \sum_{k=1}^n |\T_{k-1}|(\wh \theta_k -\theta)^t \Lambda (\wh \theta_k -\theta) = tr(\Lambda^{-1/2}L\Lambda^{-1/2}) \hspace{20pt} \text{ a.s.}\\
\end{equation}\noindent where
\begin{equation*}
\Lambda = I_2\otimes (C+D).
\end{equation*}

\end{Theo}

\noindent Our second result concerns the almost sure asymptotic properties of our least squares variance and covariance estimators $\wh \eta_n$, $\wh \zeta_n$ and $\wh \nu_n$. We need to introduce some new variables
\begin{align*}
\eta_n  &= Q_{n-1}^{-1} \sum_{k\in \T_{n-1}}  V_{2k}^2 \psi_k,\\
\zeta_n &= Q_{n-1}^{-1} \sum_{k\in \T_{n-1}}  V_{2k+1}^2 \psi_k,\\
\nu_n &= Q_{n-1}^{-1}\sum_{k\in\T_{n-1}} V_{2k}V_{2k+1}\psi_k.
\end{align*}

\begin{Theo}\label{CVpsvar}
 Assume that \H{H1} to \H{H3} are satisfied. Then, $\wh \eta_n$ and $ \wh \zeta_n$ both converge almost surely to $\eta$ and 
$\zeta$ respectively. More precisely,
\begin{align}
\|\wh \eta_{n} - \eta_{n}\|  &= \calO\left(\frac{n}{|\T_{n-1}|}\right) \hspace{20pt} \text{ a.s.}\label{vitesseeta} \\
\|\wh \zeta_{n} - \zeta_{n}\| &= \calO\left(\frac{n}{|\T_{n-1}|}\right) \hspace{20pt} \text{ a.s.}\label{vitesseetad} 
\end{align}

\noindent In addition, $\wh \nu_n$ converges almost surely to $\nu$ with
\begin{equation}
\|\wh \nu_n - \nu_n\| = \calO\left(\frac n{|\T_{n-1}|}\right) \hspace{20pt} \text{ a.s.} \label{vitesserho}
\end{equation}

\end{Theo}

\begin{Rem}\label{remrate}
We also have the almost sure rates of convergence
$$\|\wh \eta_{n} - \eta\|^2 =\calO\left(\frac{n}{|\T_{n-1}|}\right),~~ \|\wh \zeta_{n} - \zeta\|^2 =\calO\left(\frac{n}{|\T_{n-1}|}\right),~~ \|\wh \nu_{n} - \nu\|^2 =\calO\left(\frac{n}{|\T_{n-1}|}\right) ~~~ a.s.$$
\end{Rem}

\noindent Finally, our last result is devoted to the asymptotic normality of our least squares estimates $\wh\theta_n$, $\wh \eta_n$, $\wh \zeta_n$ and $\wh \nu_n$.

\begin{Theo}\label{TCL}
 Assume that \H{H1} to \H{H5} are satisfied. Then, we have the asymptotic normality
\begin{equation} \label{TCLtheta}
 \sqrt{|\T_{n-1}|}(\wh\theta_n - \theta) \liml \calN(0,\Gamma^{-1}L\Gamma^{-1}).
\end{equation}

\noindent In addition, we also have
\begin{eqnarray}
\sqrt{|\T_{n-1}|}\left(\wh \eta_{n} - \eta\right)  \liml \calN(0,A^{-1} M_{ac} A^{-1}) \label{TCLeta},\\
\sqrt{|\T_{n-1}|}\left(\wh \zeta_{n} - \zeta\right) \liml \calN(0,A^{-1} M_{bd} A^{-1}) \label{TCLetad},
\end{eqnarray}
\noindent where
\begin{equation*} \label{defD} \Gamma = I_2\otimes C, \hspace{30pt} A = \begin{pmatrix} s_4 & s_2 \\ s_2 & 1 \end{pmatrix},\end{equation*}
$$M_{ac} = (\mu_a^4-\sigma_a^4)\begin{pmatrix} s_8 & s_6 \\ s_6 & s_4 \end{pmatrix}+4\sigma_a^2\sigma_c^2\begin{pmatrix} s_6 & s_4 \\ s_4 & s_2 \end{pmatrix}+(\mu_c^4-\sigma_c^4) \begin{pmatrix} s_4 & s_2 \\ s_2 & 1 \end{pmatrix},$$
$$M_{bd} =(\mu_b^4-\sigma_b^4)\begin{pmatrix} s_8 & s_6 \\ s_6 & s_4 \end{pmatrix}+4\sigma_b^2\sigma_d^2\begin{pmatrix} s_6 & s_4 \\ s_4 & s_2 \end{pmatrix}+(\mu_d^4-\sigma_d^4) \begin{pmatrix} s_4 & s_2 \\ s_2 & 1 \end{pmatrix}.$$
Finally,
\begin{equation} 
\sqrt{|\T_{n-1}|}  \left(\wh \nu_n - \nu\right) \liml \calN\left(0,A^{-1} H A^{-1} \right) \label{TCLrho}
\end{equation}

\noindent where 
$$H = (\nu_{ab}^2-\rho_{ab}^2)\begin{pmatrix} s_8 & s_6 \\ s_6 & s_4 \end{pmatrix} + (\sigma_a^2 \sigma_d^2 + \sigma_b^2 \sigma_c^2+2\rho_{ab}\rho_{cd})\begin{pmatrix} s_6 & s_4 \\ s_4 & s_2 \end{pmatrix} + (\nu_{cd}^2 - \rho_{cd}^2) \begin{pmatrix} s_4 & s_2 \\ s_2 & 1 \end{pmatrix}.$$

\end{Theo}

\noindent The rest of the paper is dedicated to the proof of our main results.


\section{On the Rademacher-Menchov theorem}

Our almost sure convergence results rely on the Rademacher-Menchov theorem for orthonormal sequences of random variables given by Rademacher \cite{Rademacher} and Menchoff \cite{Menchoff}, see Stout \cite{Stout} and also Tandori \cite{Tandori1,Tandori2} and an unpublished note of Talagrand \cite{Talagrand} for some extensions of this result.

\begin{Theo}\label{ThRM}
Let $(X_n)$ be an orthonormal sequence of square integrable random variables which means that for all $n\neq k$,
$$\E[X_nX_k]=0\hspace{20pt} \text{and} \hspace{20pt} \E[X_n^2]=1.$$
Assume that a sequence of real numbers $(a_n)$ satisfies
\begin{equation}\label{RMhyp}
\sum_{n=1}^\infty a_n^2 (\log n)^2 < \infty.
\end{equation}
Then, the series
\begin{equation}\label{RMres}
\sum_{k=1}^\infty a_k X_k
\end{equation}
converges almost surely.
\end{Theo}

\begin{Rem}
One can observe that $(X_n)$ is neither a sequence of independent random variables nor a sequence of uncorrelated random variables since $(X_n)$ is not necessarily centered. In addition, in the case where $(X_n)$ is an orthogonal sequence of random variables, we have the same result \eqref{RMres}, replacing \eqref{RMhyp} by
$$\sum_{n=1}^\infty a_n^2 \E[X_n^2](\log n)^2 < \infty.$$
\end{Rem}
\noindent Moreover, if $a_n=1/n$, it follows from \eqref{RMres} and Kronecker's lemma that
$$\lim_{n\to\infty} \frac1n \sum_{k=1}^n X_k = 0 \as$$



\section{Proof of the keystone Lemma \ref{CVsomme}}\label{preuvepremlem}


We shall introduce some suitable notations. Let $(\beta_n)$ be the sequence defined, for all $n\geq1$, by $\beta_{2n}=a_n$ and $\beta_{2n+1}=b_n$. Then \eqref{systori} can be rewritten as
\begin{equation*} 
\begin{cases}
X_{2n} &= \beta_{2n} X_n + \veps _{2n},\\
X_{2n+1} &= \beta_{2n+1} X_n + \veps_{2n+1}.
\end{cases}
\end{equation*}
Consequently, for all $n\geq2$
$$X_n = \beta_n X_{\left[\frac n2\right]} + \veps_n.$$
\noindent First of all, let us prove that
$$\lim_{n\to\infty} \frac1{|\T_n|} L_n = s_1 \hspace{20pt} \text{where} \hspace{20pt} L_n = \sum_{k\in\T_n} X_k.$$
One can observe that
\begin{align*}
L_n &= \sum_{k\in\T_n} X_k = X_1+ \sum_{k\in\T_n\backslash \T_0} \left(\beta_k X_{\left[\frac k2\right]} + \veps_k \right),\\
&= X_1 + \sum_{k\in\T_{n-1}}(a_kX_k+\veps_{2k}+b_kX_k+\veps_{2k+1}),\\
&= X_1 + (a+b)L_{n-1} + A_{n-1}+B_{n-1}+E_{n-1},
\end{align*}
where
$$A_n = \sum_{k\in\T_{n}} X_k(a_k-a), \hspace{20pt} B_n = \sum_{k\in\T_{n}} X_k(b_k-b), \hspace{20pt} E_n = \sum_{k\in\T_{n}} (\veps_{2k}+\veps_{2k+1}).$$
Hence, we obtain
\begin{align}
\frac{L_n}{2^{n+1}} &= \frac{X_1}{2^{n+1}} + \frac{a+b}2 \frac{L_{n-1}}{2^n} + \frac{A_{n-1}}{2^{n+1}} + \frac{B_{n-1}}{2^{n+1}} + \frac{E_{n-1}}{2^{n+1}}, \nonumber\\
&= \left(\frac{a+b}2\right)^n \frac{L_0}2 + \sum_{k=1}^n \left(\frac{a+b}2\right)^{n-k}\left(\frac{X_1}{2^{k+1}} + \frac{A_{k-1}}{2^{k+1}} + \frac{B_{k-1}}{2^{k+1}} + \frac{E_{k-1}}{2^{k+1}}\right). \label{sommeL}
\end{align}
Recalling that $|\T_n| = 2^{n+1}-1$, the standard strong law of large numbers immediately implies that
$$\lim_{n\to\infty} \frac{E_{n}}{2^{n+1}} = \E[\veps_2 + \veps_3] = c+d \as$$
Let us tackle the limit of $A_n$ using the Rademacher-Menchov theorem given in Theorem \ref{ThRM}. Let $Y_n$ and $R_n$ be defined as
$$Y_n = X_n(a_n-a) \hspace{20pt} \text{and} \hspace{20pt} R_n = \sum_{k=1}^n Y_k.$$
For all $n\geq0$ and for all $k\in\G_n$,
$$\E[a_k-a|\calF_n]=\E[a_k-a]=0.$$
Moreover, we clearly have for all $n\geq2$ and for all different $k,l\in\G_n$,
\begin{align*}
\E[Y_kY_l]&=\E\left[\E[X_kX_l(a_k-a)(a_l-a)|\calF_n]\right],\\
&=\E\left[X_kX_l\E[a_k-a|\calF_n]\E[a_l-a|\calF_n]\right],\\
&=0.
\end{align*}
It means that $(Y_n)$ is a sequence of orthogonal random variables. In addition we have, for all $n\geq0$ and for all $k\in\G_n$,
\begin{align*}
\E[Y_k^2] &= \E\left[\E[X_k^2(a_k-a)^2|\calF_n]\right], \\
&= \E\left[X_k^2\E[(a_k-a)^2|\calF_n]\right]=\sigma_a^2\E[X_k^2].
\end{align*}
In order to calculate $\E[X_n^2]$, let us remark that
$$X_n = \left(\prod_{k=0}^{r_n-1}\beta_{\left[\frac n{2^k}\right]}\right)X_1 + \sum_{k=0}^{r_n-1}\left(\prod_{i=0}^{k-1}\beta_{\left[\frac n{2^i}\right]}\right)\veps_{\left[\frac n{2^k}\right]}.$$
Consequently,
\begin{multline*}
\E[X_n^2]  = \E\left[\left(\prod_{k=0}^{r_n-1}\beta_{\left[\frac n{2^k}\right]}^2\right)X_1^2\right]+\E\left[\left(\sum_{k=0}^{r_n-1}\left(\prod_{i=0}^{k-1}\beta_{\left[\frac n{2^i}\right]}\right)\veps_{\left[\frac n{2^k}\right]}\right)^2\right]\\
+2\sum_{k=0}^{r_n-1}\E\left[\left(\prod_{l=0}^{r_n-1}\beta_{\left[\frac n{2^l}\right]}\right)X_1\left(\prod_{i=0}^{k-1}\beta_{\left[\frac n{2^i}\right]}\right)\veps_{\left[\frac n{2^k}\right]}\right].
\end{multline*}
First of all,
$$\E\left[\left(\prod_{k=0}^{r_n-1}\beta_{\left[\frac n{2^k}\right]}^2\right)X_1^2\right] = \E[X_1^2]\prod_{k=0}^{r_n-1} \E\left[\beta_{\left[\frac n{2^k}\right]}^2\right]\leq\E[X_1^2] \max(\E[a_1^2],\E[b_1^2])^{r_n}\leq \E[X_1^2].$$
Next, for the cross term
\begin{align*}
&\left|\sum_{k=0}^{r_n-1}\E\left[\left(\prod_{l=0}^{r_n-1} \beta_{\left[\frac n{2^l}\right]}\right)X_1\left(\prod_{i=0}^{k-1}\beta_{\left[\frac n{2^i}\right]}\right)\veps_{\left[\frac n{2^k}\right]}\right]\right|\\
&\hspace{60pt}= \left|\sum_{k=0}^{r_n-1}\E\left[\left(\prod_{i=0}^{k-1}\beta^2_{\left[\frac n{2^i}\right]}\right)\left(\prod_{l=k+1}^{r_n-1}\beta_{\left[\frac n{2^l}\right]}\right)X_1\beta_{\left[\frac n{2^k}\right]}\veps_{\left[\frac n{2^k}\right]}\right]\right|,\\
&\hspace{60pt}=\left|\E[X_1]\sum_{k=0}^{r_n-1}\left(\prod_{i=0}^{k-1}\E\left[\beta^2_{\left[\frac n{2^i}\right]}\right]\right)\left(\prod_{l=k+1}^{r_n-1}\E\left[\beta_{\left[\frac n{2^l}\right]}\right]\right)\E\left[\beta_{\left[\frac n{2^l}\right]}\veps_{\left[\frac n{2^k}\right]}\right]\right|,\\
&\hspace{60pt}\leq \E[|X_1|] \sum_{k=0}^{r_n-1} \max(\E[a_1^2],\E[b_1^2])^{k} \max(|a|,|b|)^{r_n-k-1} \max(|ac|,|bd|),\\
&\hspace{60pt}\leq \E[X_1] \max(|ac|,|bd|) \frac{\max(|a|,|b|)^{r_n} - \max(\E[a_1^2],\E[b_1^2])^{r_n}}{\max(|a|,|b|)-\max(\E[a_1^2],\E[b_1^2])},\\
&\hspace{60pt}\leq \E[X_1] \max(|ac|,|bd|) \frac{1}{\left|\max(|a|,|b|)-\max(\E[a_1^2],\E[b_1^2])\right|}.
\end{align*}
Finally, for the last term,
\begin{align*}
&\E\left[\left(\sum_{k=0}^{r_n-1}\left(\prod_{i=0}^{k-1}\beta_{\left[\frac n{2^i}\right]}\right)\veps_{\left[\frac n{2^k}\right]}\right)^2\right]\\
&\hspace{10pt}= 2\sum_{k=1}^{r_n-1} \sum_{l=0}^{k-1} \E\left[\prod_{i=0}^{l-1}\beta_{\left[\frac{n}{2^i}\right]}^2\prod_{j=l+1}^{k-1}\beta_{\left[\frac{n}{2^j}\right]}\beta_{\left[\frac{n}{2^l}\right]}\veps_{\left[\frac{n}{2^l}\right]}\veps_{\left[\frac{n}{2^k}\right]}\right]+\sum_{k=0}^{r_n-1}\E\left[\prod_{i=0}^{k-1}\beta_{\left[\frac{n}{2^i}\right]}^2\veps_{\left[\frac{n}{2^k}\right]}^2\right],\\
&\hspace{10pt}\leq 2\sum_{k=1}^{r_n-1} \sum_{l=0}^{k-1} \max(\E[a_1^2],\E[b_1^2])^l \max(|a|,|b|)^{k-l-1}\max(|ac|,|bd|) \max(|c|,|d|)^2 \\
&\hspace{200pt} + \sum_{k=0}^{r_n-1} \max(\E[a_1^2],\E[b_1^2])^k \max(\E[\veps_2^2],\E[\veps_3^2]),\\
&\hspace{10pt} \leq 2 \frac{\max(|ac|,|bd|)\max(|c|,|d|)}{\left|\max(|a|,|b|)-\max(\E[a_1^2],\E[b_1^2])\right|}\left(\frac{1}{1-\max(|a|,|b|)} + \frac{1}{1-\max(\E[a_1^2],\E[b_1^2])}\right)\\
&\hspace{290pt} + \frac{\max(\E[\veps_2^2],\E[\veps_3^2])}{1-\max(\E[a_1^2],\E[b_1^2])}.
\end{align*}
To sum up, we proved that it exists some positive constant $\mu$ such that, for all $n\geq0$, $\E[X_n^2]\leq \mu$, leading to 
$$\sum_{n=1}^\infty\frac1{n^2}\E[Y_n^2](\log n)^2\leq \sigma_a^2\mu\sum_{n=1}^\infty \frac{(\log n)^2}{n^2}<\infty.$$
Therefore, it follows from the Rademacher-Menchov theorem that the series
$$\sum_{k=1}^n\frac1k Y_k$$
converges a.s. Consequently, Kronecker's lemma implies that
$$\lim_{n\to\infty} \frac1n \sum_{k=1}^n Y_n = \lim_{n\to\infty} \frac1n R_n = 0 \as$$
In particular
$$\lim_{n\to\infty} \frac1{|\T_n|} R_{|\T_n|} = \lim_{n\to\infty} \frac1{|\T_n|} A_n = 0 \as$$
Hence, we find that
$$\lim_{n\to\infty} \frac1{2^{n+1}} A_n = 0 \as$$
By the same token, we also have
$$\lim_{n\to\infty} \frac1{2^{n+1}} B_n = 0 \as$$
To sum up, we obtain that
\begin{equation}\label{lim1}
\lim_{n\to\infty} \frac{X_1}{2^{n+1}} + \frac{A_{n-1}}{2^{n+1}} + \frac{B_{n-1}}{2^{n+1}} + \frac{E_{n-1}}{2^{n+1}} = \frac{c+d}2 \as
\end{equation}
Therefore, we deduce from \eqref{sommeL} and \eqref{lim1} together with the assumption that $|a|<1$ and $|b|<1$, that
\begin{equation}\label{lim2}
\lim_{n\to\infty} \frac{L_n}{2^{n+1}} = \frac{c+d}2\frac1{1-\displaystyle\frac{a+b}2} \as
\end{equation}
leading to
$$\lim_{n\to\infty} \frac1{|\T_n|} \sum_{k\in\T_n} X_k = \frac{c+d}{2-(a+b)} \as$$
Let us now tackle the study of
$$K_n = \sum_{k\in\T_n} X_k^2.$$
First, one can observe that
\begin{align*}
K_n &= \sum_{k\in\T_n} X_k^2 = X_1^2+ \sum_{k\in\T_n\backslash \T_0} \left(\beta_k X_{\left[\frac k2\right]} + \veps_k\right)^2,\\
&= X_1^2+ \left(\sum_{k\in\T_n\backslash \T_0}\beta_k^2 X_{\left[\frac k2\right]}^2\right) + 2\left(\sum_{k\in\T_n\backslash \T_0}\beta_k\veps_kX_{\left[\frac k2\right]}\right) + \left(\sum_{k\in\T_n\backslash \T_0}\veps_k^2\right),\\
&= X_1^2 + (\sigma_a^2 +\sigma_b^2 + a^2 +b^2) K_{n-1} + 2(ac+bd)L_{n-1} + A_{n-1} + B_{n-1} + E_{n-1},
\end{align*}
where
$$A_n = \sum_{k\in\T_n} X_k^2(a_k^2 + b_k^2 -(\sigma_a^2+\sigma_b^2+a^2+b^2)),$$
$$B_n = \sum_{k\in\T_n} X_k(a_k\veps_{2k} + b_k\veps_{2k+1} -(ac+bd)),$$
$$E_n = \sum_{k\in\T_n} (\veps_{2k}^2 + \veps_{2k+1}^2).$$
Hence we obtain, as for $L_n$
\begin{equation*}\label{sommeK}
\frac{K_n}{2^{n+1}} = \mu^n \frac{K_0}2 + \sum_{k=1}^n \mu^{n-k}\left(\frac{X_1^2}{2^{k+1}} + \nu \frac{L_{k-1}}{2^{k}} + \frac{A_{k-1}}{2^{k+1}} + \frac{B_{k-1}}{2^{k+1}} + \frac{E_{k-1}}{2^{k+1}}\right),
\end{equation*}
where, since $\E[a_k^2] = \sigma_a^2 +a^2 <1$ and $\E[b_k^2] = \sigma_b^2 +b^2 <1$,
$$\mu = \frac{\sigma_a^2 + \sigma_b^2+a^2 +b^2}{2} < 1 \hspace{20pt} \text{and} \hspace{20pt} \nu = ac+bd.$$
As previously, the strong law of large numbers leads to 
\begin{equation}\label{lim3}
\lim_{n\to\infty} \frac1{|\T_n|} E_n = \sigma_c^2 + \sigma_d^2 + c^2 +d^2 \as
\end{equation}
Moreover, it follows once again from the Rademacher-Menchov theorem with Kronecker's lemma, \eqref{lim2} and \eqref{lim3} that
\begin{equation*}\lim_{n\to_\infty} \frac{K_n}{2^{n+1}} \\= \frac1{1-\mu}\left(\nu\frac{c+d}{2-(a+b)} + \frac{\sigma_c^2 + \sigma_d^2 + c^2 +d^2}2\right) \as\end{equation*}
leading to convergence \eqref{Lemfonda} for $p=2$. We shall not carry out the proof of \eqref{Lemfonda} for $3\leq p\leq8$ inasmuch as it follows essentially the same lines that those for $p=2$. One can observe that, in order to prove \eqref{Lemfonda} for $3\leq p \leq 8$, it is necessary to assume that $\E[a_1^{2p}]<1$, $\E[b_1^{2p}]<1$, $\E[\veps_{2}^{2p}]<\infty$ and $\E[\veps_{3}^{2p}]<\infty$. The limiting values $s_3$ to $s_{8}$ may be explicitly calculated. More precisely, for all  $p\in\{1,2,\hdots,8\}$, denote
$$A_p = \E[a_1^p], \hspace{20pt} B_p = \E[b_1^p], \hspace{20pt} C_p = \E[\veps_2^p], \hspace{20pt}  D_p = \E[\veps_{3}^p].$$
We already saw that
$$s_1 = \frac{C_1+D_1}{2-(A_1+B_1)},$$
$$s_2 = \frac{2}{2-(A_2+B_2)}\left((A_1C_1+B_1D_1)s_1+\frac{C_2+D_2}2\right).$$
The other limiting values $s_3$ to $s_{8}$ of convergence \eqref{Lemfonda} can be recursively calculated via the linear relation
$$s_p = \frac{2}{2-(A_p+B_p)}\left(\sum_{k=1}^{p-1} \frac12 \binom pk (A_k C_{p-k} + B_k D_{p-k})s_k + \frac{C_p+D_p}2 \right).$$


\section{Proof of Proposition \ref{cvcrochet}}


The almost sure convergence \eqref{limcrochet} is immediate  through \eqref{defcrochetM}, \eqref{defLk} and Lemma \ref{CVsomme}. Let us now prove that $L$ is a positive definite matrix. First, the matrices
$$\begin{pmatrix} \sigma_a^2 & \rho_{ab} \\ \rho_{ab} & \sigma_b^2 \end{pmatrix} \hspace{20pt} \text{and} \hspace{20pt} \begin{pmatrix} \sigma_c^2 & \rho_{cd} \\ \rho_{cd} & \sigma_d^2 \end{pmatrix}$$
are clearly positive semidefinite and positive definite under \H{H3}. Moreover, $D$ is clearly positive semidefinite since
$$\lim_{n\to\infty}\frac1{|\T_n|} \sum_{k\in\T_n} \begin{pmatrix} X_k^4 & X_k^3 \\ X_k^3 & X_k^2 \end{pmatrix}= D \as$$
Finally, let us prove that $C$ is positive definite. Its trace is clearly greater than 1, hence we only have to prove that its determinant is positive. Its determinant is given by
\begin{align*}
s_2 - s_1^2 &= \frac2{2-\left(\sigma_a^2 + \sigma_b^2+a^2 +b^2\right)}\left(\frac{(ac+bd)(c+d)}{2-(a+b)} + \frac{\sigma_c^2 + \sigma_d^2 + c^2 +d^2}2\right) \\
	&\hspace{280pt}- \left(\frac{c+d}{2-(a+b)}\right)^2,\\
&=\frac{\sigma_c^2+\sigma_d^2}{2-\left(\sigma_a^2 + \sigma_b^2+a^2 +b^2\right)} + \left(\frac{c+d}{2-(a+b)}\right)^2\frac{\sigma_a^2+\sigma_b^2}{2-\left(\sigma_a^2 + \sigma_b^2+a^2 +b^2\right)}\\
	&\hspace{149pt} +\frac2{2-\left(\sigma_a^2 + \sigma_b^2+a^2 +b^2\right)} \frac{(ad-bc+c-d)^2}{(2-(a+b))^2}.
\end{align*}
The first term of this sum is positive since under \H{H1} $\sigma_a^2 + \sigma_b^2+a^2 +b^2<2$ and since under \H{H2} $\sigma_c^2+\sigma_d^2>0$. Moreover, the two other terms are clearly nonnegative, which proves that this matrix is positive definite. Since the Kronecker product of two positive semidefinite (respectively positive definite) matrices is a positive semidefinite (respectively positive definite) matrix, we can conclude that $L$ is positive definite.


\section{Proofs of the almost sure convergence results}\label{preuveCVpstheta}


We shall make use of a martingale approach, as the one developed bu Bercu et al.~\cite{BercuBDSAGP} or de Saporta et al.~\cite{BDSAGPMarsalleRCBAR}. For all $n\geq1$, let 
$$\calV_n = M_n^t P_{n-1}^{-1} M_n = (\wh\theta_n -\theta) \Sigma_{n-1} P_{n-1}^{-1} \Sigma_{n-1} (\wh \theta_n - \theta)$$
where 
$$P_n = \sum_{k\in\T_n} (1+X_k^2)I_2 \otimes \begin{pmatrix}X_k^2 & X_k \\ X_k & 1 \end{pmatrix}.$$
By the same calculations as in \cite{BercuBDSAGP}, we can easily see that
\begin{equation}\label{egaliteVn}
\calV_{n+1} + \calA_n = \calV_1 + \calB_{n+1} + \calW_{n+1},
\end{equation}
where
$$\calA_{n} = \sum_{k=1}^n M_k^t(P_{k-1}^{-1} - P_k^{-1})M_k,$$
$$B_{n+1} = 2 \sum_{k=1}^n M_k^t P_k^{-1} \Delta M_{k+1} \hspace{20pt} \text{and} \hspace{20pt} \calW_{n+1} = \sum_{k=1}^n \Delta M_{k+1}^t P_k^{-1} \Delta M_{k+1}.$$

\begin{Lem} \label{lemCVVnAn}
 Assume that \H{H1} to \H{H3} are satisfied. Then, we have
\begin{equation} \label{CVWn}
 \lim_{n\to\infty} \frac{\calW_{n}}{n} = \frac12 tr((I_2 \otimes (C+D))^{-1/2} L (I_2 \otimes (C+D))^{-1/2}) \hspace{20pt} \text{ a.s.}
\end{equation}
where $C$ and $D$ are the matrices given by \eqref{defCD}. In addition, we also have
\begin{equation}\label{CVBn}
\calB_{n+1} = o(n) \hspace{20pt} \text{ a.s.}
\end{equation}
and
\begin{equation} \label{CVVnAn}
 \lim_{n\to\infty} \frac{\calV_{n+1} + \calA_n}{n} = \frac12 tr((I_2 \otimes (C+D))^{-1/2} L (I_2 \otimes (C+D))^{-1/2}) \hspace{20pt} \text{ a.s.}
\end{equation}
\end{Lem}

\begin{proof}
\noindent First of all, we have $\calW_{n+1} = \calT_{n+1} + \calR_{n+1}$ where
$$\calT_{n+1} = \sum_{k=1}^n \frac{\Delta M_{k+1}^t (I_2 \otimes (C+D))^{-1} \Delta M_{k+1}}{|\T_k|},$$
$$\calR_{n+1} = \sum_{k=1}^n \frac{\Delta M_{k+1}^t (|\T_k| P_k^{-1} - (I_2 \otimes (C+D))^{-1}) \Delta M_{k+1}}{|\T_k|}.$$

\noindent One can observe that $\calT_{n+1} = tr((I_2 \otimes (C+D))^{-1/2} \calH_{n+1} (I_2 \otimes (C+D))^{-1/2})$ where
$$ \calH_{n+1} = \sum_{k=1}^n \frac{\Delta M_{k+1} \Delta M_{k+1}^t}{|\T_k|}.$$

\noindent Our aim is to make use of the strong law of large numbers for martingale transforms, so we start by adding and subtracting a term involving 
the conditional expectation of $\Delta \calH _{n+1}$ given $\calF_n$. We have thanks to relation \eqref{defcrochetM} that for all 
$n\geq0$, $\E[\Delta M_{n+1} \Delta M_{n+1}^t | \calF_n] = L_n$. Consequently, we can split $\calH_{n+1}$ into two terms
$$\calH_{n+1} = \sum_{k=1}^n \frac{L_k}{|\T_k|} + \calK_{n+1},$$
where 
$$\calK_{n+1} = \sum_{k=1}^n \frac{\Delta M_{k+1} \Delta M_{k+1}^t - L_k}{|\T_k|}.$$
It clearly follows from convergence \eqref{limcrochet} that 
$$ \lim_{n\to\infty} \frac{L_n}{|\T_n|} = \frac12 L \hspace{20pt} \text{ a.s.}$$
Hence, Cesaro convergence theorem immediately implies that
\begin{equation} \label{CesaroLk}
 \lim_{n\to\infty} \frac1n \sum_{k=1}^n \frac{L_k}{|\T_k|} = \frac12 L \hspace{20pt} \text{ a.s.}
\end{equation}
On the other hand, the sequence $(\calK_n)_{n\geq2}$ is obviously a square integrable martingale. Moreover, we have 
$$\Delta \calK_{n+1}= \calK_{n+1} - \calK_n = \frac1{|\T_n|} (\Delta M_{n+1} \Delta M_{n+1}^t - L_n).$$
For all $u\in\R^4$, denote $\calK_n(u) = u^t\calK_n u$. It follows from tedious but straightforward calculations, together with Lemma 
\ref{CVsomme}, that the increasing process of the martingale $(\calK_n(u))_{n\geq2}$ satisfies $\langle \calK(u)\rangle_n =\calO(n)$ a.s. Therefore, we deduce from 
the strong law of large numbers for martingales that for all $u\in\R^4$, $\calK_n(u) = o(n)$ a.s.~leading to $\calK_n = o(n)$ a.s. Hence, 
we infer from \eqref{CesaroLk} that
\begin{equation} \label{CVHn}
 \lim_{n\to\infty} \frac{\calH_{n+1}}n = \frac12 L \hspace{20pt} \text{a.s.}
\end{equation}

\noindent Via the same arguments as in the proof of convergence \eqref{limcrochet}, we find that
\begin{equation}\label{CVPn}
\lim_{n\to\infty} \frac{P_n}{|\T_n|} = I_2 \otimes (C+D) \hspace{20pt} \text{a.s.}
\end{equation}
Then, we obtain from \eqref{CVHn} that
$$ \lim_{n\to\infty} \frac{\calT_n}n = \frac12 tr((I_2 \otimes (C+D))^{-1/2} L (I_2 \otimes (C+D))^{-1/2}) \hspace{20pt} \text{a.s.}$$
which allows us to say that $\calR_n = o(n)$ a.s. leading to \eqref{CVWn}. We are now in position to prove \eqref{CVBn}. Let us recall that 
$$\calB_{n+1} = 2 \sum_{k=1}^n M_k^t P_k^{-1} \Delta M_{k+1} = 2 \sum_{k=1}^n M_k^t P_k^{-1} \Psi_k \xi_{k+1}.$$

\noindent Hence, $(\calB_n)$ is a square integrable martingale. In addition, we have
$$\Delta \calB_{n+1} = 2M_n^t P_n^{-1}\Delta M_{n+1}.$$
Consequently,
\begin{align*}
 \E[(\Delta \calB_{n+1})^2 | \calF_n] &= 4 \E[M_n^t P_n^{-1}\Delta M_{n+1} \Delta M_{n+1}^t P_n^{-1}M_n | \calF_n] \hspace{20pt} \text{a.s.} \\
	    &= 4 M_n^tP_n^{-1}\E[\Delta M_{n+1} \Delta M_{n+1}^t| \calF_n]P_n^{-1}M_n \hspace{20pt} \text{a.s.}\\
	    &= 4 M_n^tP_n^{-1} L_n P_n^{-1}M_n \hspace{20pt} \text{a.s.}
\end{align*}
However, we already saw from \eqref{defLk} that
\begin{equation*}
 L_n = \sum_{k\in\G_n} \begin{pmatrix}
                        P(X_k) & Q(X_k) \\
			Q(X_k) & R(X_k)
                       \end{pmatrix}
			      \otimes \begin{pmatrix}
			               X_k^2 & X_k \\ X_k & 1
			              \end{pmatrix}.
\end{equation*}
Moreover,
\begin{equation*}
\Delta P_n = P_n-P_{n-1} = \sum_{k\in\G_n} (1+X_k)^2 I_2 \otimes \begin{pmatrix}
						X_k^2 & X_k \\ X_k & 1
					      \end{pmatrix}.
\end{equation*}
For $\alpha=\max(\sigma_a^2,\sigma_c^2)+\max(\sigma_b^2,\sigma_d^2)+\max(|\rho_{ab}|,|\rho_{cd}|)$, denote
$$\Delta_n=\begin{pmatrix} \alpha(1+X_n^2) - P(X_n) & - Q(X_n)\\ - Q(X_n) & \alpha(1+X_n^2) - R(X_n) \end{pmatrix}$$
where $P(X_n)$, $Q(X_n)$ and $R(X_n)$ are given by \eqref{defPQR}. It is not hard to see that
$$\alpha\Delta P_n - L_n= \sum_{k\in \G_n} \Delta_k \otimes \begin{pmatrix} X_k^2 & X_k \\ X_k & 1 \end{pmatrix}.$$
We claim that $\Delta_n$ is a positive definite matrix. As a matter of fact, we deduce from the elementary inequalities
\begin{equation}\label{majpoly}
\begin{cases}
0<P(X)\leq \max(\sigma_a^2,\sigma_c^2)(1+X^2),\\
0<R(X)\leq \max(\sigma_b^2,\sigma_d^2)(1+X^2),\\
|Q(X)| \leq \max(|\rho_{ab}|,|\rho_{cd}|)(1+X^2),
\end{cases}
\end{equation}
that
\begin{align*}
tr(\Delta_n) &= 2\alpha(1+X_k^2) - P(X_n) -R(X_n) \\ &\geq (2\alpha - \max(\sigma_a^2,\sigma_c^2)-\max(\sigma_b^2,\sigma_d^2))(1+X_k^2) >0.
\end{align*}
In addition, we also have from \eqref{majpoly} that
\begin{align*}
\det(\Delta_n) &= (\alpha (1+X_n^2) - P(X_n))(\alpha (1+X_n^2) - R(X_n)) - Q^2(X_n),\\
	&=\alpha (1+X_n^2)\left(\alpha (1+X_n^2) - P(X_n) - R(X_n)\right) + P(X_n)R(X_n) - Q^2(X_n),\\
	&\geq P(X_n)R(X_n) + \alpha (1+X_n^2)^2 \max(|\rho_{ab}|,|\rho_{cd}|) - Q^2(X_n),\\
	&\geq P(X_n)R(X_n) + \max(|\rho_{ab}|,|\rho_{cd}|)^2 (1+X_n^2)^2 - Q^2(X_n) > 0.
\end{align*}
Consequently, $\Delta_n$ is a positive definite matrix which immediately implies that $L_n \leq \alpha \Delta P_n$. Moreover, we can use Lemma B.1 of \cite{BercuBDSAGP} to say that 
$$P_{n-1}^{-1} \Delta P_n P_{n-1}^{-1} \leq P_{n-1}^{-1} - P_n^{-1}.$$

\noindent Hence
\begin{align*}
 \E[(\Delta \calB_{n+1})^2 | \calF_n] &= 4 M_n^tP_n^{-1} L_n P_n^{-1}M_n \hspace{20pt} \text{a.s.}\\
			&\leq 4 \alpha M_n^tP_n^{-1} \Delta P_n P_n^{-1}M_n \hspace{20pt} \text{a.s.}\\
			&\leq 4 \alpha M_n^t(P_{n-1}^{-1} - P_n^{-1})M_n \hspace{20pt} \text{a.s.}\\
\end{align*}

\noindent leading to $\langle\calB\rangle_n \leq 4\alpha \calA_n$. Therefore it follows from the strong law of large numbers for martingales that $\calB_n = o(\calA_n)$. Hence, we deduce from decomposition \eqref{egaliteVn} that 
$$\calV_{n+1} + \calA_n = o(\calA_n) + \calO(n) \hspace{20pt} \text{ a.s.}$$

\noindent leading to $\calV_{n+1} = \calO(n)$ and $\calA_n = \calO(n)$ a.s.~which implies that $\calB_n= o(n)$ a.s. Finally we clearly obtain convergence \eqref{CVVnAn} from the main decomposition \eqref{egaliteVn} together with \eqref{CVWn} and \eqref{CVBn}, which completes the proof of Lemma \ref{lemCVVnAn}.
\end{proof}

\begin{Lem}\label{delta}
 Assume that \H{H1} to \H{H3} are satisfied. For all $\delta > 1/2$, we have
\begin{equation*}\label{majdelta}\|M_n\|^2 = o(|\T_n|n^\delta) \hspace{20pt} \text{ a.s.}\end{equation*}
\end{Lem}

\begin{proof}
Let us recall that
$$M_n = \sum_{k\in\T_{n-1}} \begin{pmatrix} X_kV_{2k} \\ V_{2k} \\ X_kV_{2k+1} \\V_{2k+1}\end{pmatrix}.$$
Denote
$$\begin{array}{ccccc}T_n = \displaystyle \sum_{k\in\T_{n-1}} X_kV_{2k}  & & \text{ and } & & \displaystyle U_n = \sum_{i\in\T_{n-1}} V_{2k}.\end{array}$$
On the one hand, $T_n$ can be rewritten as
$$\begin{array}{ccccc}\displaystyle{ T_n = \sum_{k=1}^n \sqrt{|\G_{k-1}|} f_k } & & \text{ where } & & \displaystyle f_n = \frac1{\sqrt{|\G_{n-1}|}}\sum_{k\in\G_{n-1}} X_kV_{2k}.\end{array}$$
We already saw in Section 3 that, for all $n\geq0$ and for all $k\in\G_n$,
$$\begin{array}{ccccccc} \E[V_{2k}|\calF_n] = 0 & & \text{ and } & & \E[V_{2k}^2|\calF_n] = \sigma_a^2X_k^2+\sigma_c^2 = P(X_k) & & \text{a.s.} \end{array}$$
In addition, for all $k\in\G_n$,
$$\E[V_{2k}^4|\calF_n] = \mu_a^4X_k^4 + 6\sigma_a^2\sigma_c^2 X_k^2 + \mu_c^4 \hspace{20pt} \text{a.s.}$$
which implies that
\begin{equation}\label{majV2k4} \E[V_{2k}^4|\calF_n] \leq \mu_{ac}^4 (1+X_k^2)^2 \hspace{20pt} \text{a.s.}\end{equation}
where $\mu_{ac}^4 = \max(\mu_a^4,3\sigma_a^2\sigma_c^2,\mu_c^4)$. Consequently, $\E[f_{n+1}|\calF_n]=0$ a.s.~and we deduce from \eqref{majV2k4} together with the Cauchy-Schwarz inequality that
\begin{align*}
\E[f_{n+1}^4|\calF_n] &= \frac1{|\G_n|} \E\left[\left. \(\sum_{k\in\G_n} X_kV_{2k} \)^4 \right|\calF_n\right],\\
	&= \frac1{|\G_n|^2} \sum_{k\in\G_n}X_k^4\E[V_{2k}^4|\calF_n]\\
		&\hspace{80pt}+ \frac3{|\G_n|^2} \sum_{k\in\G_n}\sum_{\substack{l\in\G_n \\ l\neq k}}X_k^2X_l^2\E[V_{2k}^2|\calF_n] \E[V_{2l}^2|\calF_n],
\end{align*}
leading to
$$\E[f_{n+1}^4|\calF_n]\leq \frac{\mu_{ac}^4}{|\G_n|^2} \sum_{k\in\G_n} X_k^4(1+X_k^2)^2 + 3\max(\sigma_a^2,\sigma_c^2)^2 \(\frac{1}{|\G_n|}\sum_{k\in\G_n}X_k^2(1+X_k^2)\)^2.$$
Therefore, we infer from Lemma \ref{CVsomme} that
$$\sup_{n\geq0}\E[f_{n+1}^4|\calF_n]<\infty \hspace{20pt} \text{a.s.}$$
Hence, we obtain from Wei's lemma given in \cite{Wei} page 1672 that for all $\delta > 1/2$,
$$T_n^2 = o(|\T_{n-1}|n^\delta) \hspace{20pt} \text{a.s.}$$
On the other hand, $U_n$ can be rewritten as
$$\begin{array}{ccccc}\displaystyle{ U_n = \sum_{k=1}^n \sqrt{|\G_{k-1}|} g_k } & & \text{ where } & & \displaystyle g_n = \frac1{\sqrt{|\G_{n-1}|}}\sum_{k\in\G_{n-1}} V_{2k}.\end{array}$$
Via the same calculation as before, $\E[g_{n+1}|\calF_n] = 0$ a.s.~and
$$\E[g_{n+1}^4|\calF_n] \leq \frac{\mu_{bd}^4}{|\G_n|^2} \sum_{k\in\G_n}  (1+X_k^2)^2 + 3\max(\sigma_b^2,\sigma_d^2)^2 \(\frac{1}{|\G_n|}\sum_{k\in\G_n}(1+X_k^2)\)^2$$
Hence, we deduce once again from Lemma \ref{CVsomme} and Wei's Lemma that for all $\delta > 1/2$,
$$U_n^2 = o(|\T_{n-1}|n^\delta) \hspace{20pt} \text{a.s.}$$
In the same way, we obtain the same result for the two last components of $M_n$, which completes the proof of Lemma \ref{delta}.
\end{proof}

\subsection{Proof of Theorem \ref{CVpstheta}.} We recall that $\calV_n = (\wh\theta_n -\theta) \Sigma_{n-1} P_{n-1}^{-1} \Sigma_{n-1} (\wh \theta_n - \theta)$ which implies that
$$\|\wh \theta_n - \theta \|^2 \leq \frac{\calV_n}{\lambda_{min}(\Sigma_{n-1}P_{n-1}^{-1}\Sigma_{n-1})}.$$
On the one hand, it follows from \eqref{CVVnAn} that $\calV_n=\calO(n)$ a.s. On the other hand, we deduce from Lemma \ref{CVsomme} that
\begin{equation}\lim_{n\to\infty} \frac{\Sigma_n}{|\T_n|} = I_2\otimes C \as \label{CVSigman}\end{equation}
where $C$ is the positive definite matrix given by \eqref{defCD}. Therefore, we obtain from \eqref{CVPn} and \eqref{egaliteVn} that
$$\lim_{n\to\infty} \frac{\lambda_{min}(\Sigma_{n-1}P_{n-1}^{-1}\Sigma_{n-1})}{|\T_{n-1}|} = \lambda_{min}(C(C+D)^{-1}C) >0 \hspace{20pt} \text{a.s.}$$
Consequently, we find that
$$\|\wh \theta_n - \theta \|^2 = \calO\left(\frac{n}{|\T_{n-1}|}\right) \hspace{20pt} \text{a.s.}$$
We are now in position to prove the quadratic strong law \eqref{quadratic1}. First of, all a direct application of Lemma \ref{delta} ensures that $\calV_n = o(n^\delta)$ a.s.~for all $\delta > 1/2$. Hence, we obtain from \eqref{CVVnAn} that
\begin{equation}\lim_{n\to\infty} \frac{\calA_n}n = \frac12 tr((I_2 \otimes (C+D))^{-1/2} L (I_2 \otimes (C+D))^{-1/2}) \hspace{20pt} \text{ a.s.}\label{CVAn/n}\end{equation}

\noindent Let us rewrite $\calA_n$ as 
$$\calA_n = \sum_{k=1}^n M_k^t\left(P_{k-1}^{-1} - P_k^{-1}\right) M_k = \sum_{k=1}^n M_k^tP_{k-1}^{-1/2} A_k P_{k-1}^{-1/2} M_k,$$

\noindent where $A_k = I_4 - P_{k-1}^{1/2}P_{k}^{-1}P_{k-1}^{1/2}$. We already saw from \eqref{CVPn} that 
$$\lim_{n\to\infty} \frac{P_n}{|\T_n|} = I_2\otimes (C+D) \hspace{20pt} \text{a.s.}$$
which ensures that
$$\displaystyle \lim_{n\to\infty} A_n = \frac12 I_4 \hspace{20pt} \text{a.s.}$$
In addition, we deduce from \eqref{CVVnAn} that $\calA_n=\calO(n)$ a.s.~which implies that
\begin{equation}\frac{\calA_n}n = \left(\frac1{2n}\sum_{k=1}^nM_k^tP_{k-1}^{-1}M_k\right) + o(1) \hspace{20pt} \text{ a.s.}\label{eqAn/n}\end{equation}
Moreover, we also have
\begin{align}
 \frac1n\sum_{k=1}^nM_k^tP_{k-1}^{-1}M_k &= \frac1n \sum_{k=1}^n (\wh \theta_k - \theta)^tP_{k-1}(\wh \theta_k - \theta), \nonumber\\
	  &= \frac1n \sum_{k=1}^n |\T_{k-1}| (\wh \theta_k - \theta)^t\frac{P_{k-1}}{|\T_{k-1}|}(\wh \theta_k - \theta), \nonumber\\
	  &= \frac1n \sum_{k=1}^n |\T_{k-1}| (\wh \theta_k - \theta)^t(I_2 \otimes (C+D))(\wh \theta_k - \theta) + o(1) \hspace{20pt} \text{ a.s.} \label{elementQSL}
\end{align}

\noindent Therefore, \eqref{CVAn/n} together with \eqref{eqAn/n} and \eqref{elementQSL} lead to \eqref{quadratic1}.


\subsection{Proof of Theorem \ref{CVpsvar}}


We only prove \eqref{vitesseeta} inasmuch as the proof of \eqref{vitesseetad} follows exactly the same lines. Relation \eqref{esteta} immediately leads to
\begin{align}
Q_{n-1}(\wh \eta_n - \eta_n) &= \sum_{l=0}^{n-1} \sum_{k\in\G_l} (\wh V_{2k}^2 - V_{2k}^2) \psi_k,\nonumber\\
	&= \sum_{l=0}^{n-1} \sum_{k\in\G_l} \left((\wh V_{2k} - V_{2k})^2 +2(\wh V_{2k} - V_{2k})V_{2k}\right)\psi_k.\label{diffeta}
\end{align}
Moreover, we clearly have from Section \ref{LSE} that, for all $n\geq0$ and for all $k\in\G_n$
$$\wh V_{2k} - V_{2k} = -\begin{pmatrix} \wh a_n -a \\ \wh c_n -c \end{pmatrix}^t \Phi_k,$$
which implies that
$$(\wh V_{2k} - V_{2k})^2 \leq \left((\wh a_n - a)^2 + (\wh c_n - c)^2\right)\|\Phi_k\|^2 = \left((\wh a_n - a)^2 + (\wh c_n - c)^2\right)(1+X_k^2).$$
In addition, since $\|\psi_k\|^2 = 1+X_k^4 \leq (1+X_k^2)^2$, we have
\begin{align*}
\left\|\sum_{l=0}^{n-1} \sum_{k\in\G_l} (\wh V_{2k} - V_{2k})^2 \psi_k\right\| &\leq \sum_{l=0}^{n-1} \left((\wh a_l - a)^2 + (\wh c_l - c)^2\right) \sum_{k\in\G_l} (1+X_k^2)^2.
\end{align*}
However, it follows from Lemma \ref{CVsomme} that
$$\sum_{k\in\G_l} (1+X_k^2)^2 = \calO(|\G_l|) \as$$
and since $\Lambda$ is positive definite, \eqref{quadratic1} leads to 
$$\sum_{l=0}^{n-1} \left((\wh a_l - a)^2 + (\wh c_l - c)^2\right) |\G_l| = \calO(n) \as$$
Hence, we find that 
\begin{equation}
\left\|\sum_{l=0}^{n-1} \sum_{k\in\G_l} (\wh V_{2k} - V_{2k})^2 \psi_k\right\|  = \calO(n) \as \label{premierterme}
\end{equation}
Let us now tackle
$$P_n = \sum_{l=0}^{n-1} \sum_{k\in\G_l} (\wh V_{2k} - V_{2k})V_{2k}\psi_k.$$
It is clear that
\begin{align*}
\Delta P_{n+1} &= P_{n+1} - P_n = \sum_{k\in\G_n} (\wh V_{2k} - V_{2k})V_{2k}\psi_k,\\
&= -\sum_{k\in\G_n} V_{2k}\psi_k\Phi_k^t\begin{pmatrix} \wh a_n -a \\ \wh c_n -c \end{pmatrix}.
\end{align*}
Since, for al $k\in\G_n$, $\E[V_{2k}|\calF_n]=0$ a.s. and $\E[V_{2k}^2|\calF_n] = P(X_k)$ a.s., we have 
\begin{align*}
\E[\Delta P_{n+1} \Delta P_{n+1}^t | \calF_n] = \sum_{k\in\G_n}P(X_k) \psi_k \Phi_k^t \begin{pmatrix} \wh a_n -a \\ \wh c_n -c \end{pmatrix}\begin{pmatrix} \wh a_n -a \\ \wh c_n -c \end{pmatrix}^t \Phi_k \psi_k^t \as
\end{align*}
which allows to say that $(P_n)$ is a square integrable martingale with increasing process $\langle P \rangle_n$ given by
\begin{align*}
\langle P \rangle_n &= \sum_{l=0}^{n-1} \E[\Delta P_{l+1} \Delta P_{l+1}^t | \calF_n],\\
	&= \sum_{l=0}^{n-1} \sum_{k\in\G_l}P(X_k) \psi_k \Phi_k^t \begin{pmatrix} \wh a_l -a \\ \wh c_l -c \end{pmatrix}\begin{pmatrix} \wh a_l -a \\ \wh c_l -c \end{pmatrix}^t \Phi_k \psi_k^t \as
\end{align*}
Consequently, if $\alpha=\max(\sigma_a^2,\sigma_c^2)$, we obtain that
\begin{align*}
\|\langle P \rangle_n\| &\leq \alpha \sum_{l=0}^{n-1} \left((\wh a_l - a)^2 + (\wh c_l - c)^2\right) \sum_{k\in\G_l} (1+X_k^2) \|\psi_k\|^2 \|\Phi_k\|^2 \as\\
	&\leq \alpha \sum_{l=0}^{n-1} \left((\wh a_l - a)^2 + (\wh c_l - c)^2\right) \sum_{k\in\G_l} (1+X_k^2)^4 \as
\end{align*}
leading, as previously via Lemma \ref{CVsomme} and \eqref{quadratic1}, to
$$\|\langle P \rangle_n\| = \calO(n) \as$$
The strong law of large numbers for martingale given e.g.~in Theorem 1.3.15 of \cite{Duflo} implies that
\begin{equation}
P_n = o(n) \as \label{deuxiemeterme}
\end{equation}
Then, we deduce from \eqref{diffeta}, \eqref{premierterme} and \eqref{deuxiemeterme} that
\begin{equation}\label{Oden}\|Q_{n-1}(\wh \eta_n - \eta_n)\| = \calO(n) \as\end{equation}
Moreover, we obtain through Lemma \ref{CVsomme} that
\begin{equation}
\lim_{n\to\infty} \frac1{|\T_n|} Q_n = \begin{pmatrix} s_4 & s_2 \\ s_2 & 1 \end{pmatrix} \as \label{limQ}
\end{equation}
and we can prove, through tedious calculations, that this limiting matrix is positive definite. Therefore, \eqref{Oden} immediately implies \eqref{vitesseeta}. We shall now proceed to the proof of \eqref{vitesserho}. Denote
$$R_n = \sum_{k\in\T_{n-1}} (\wh W_k - W_k)^t J W_k \psi_k,$$
where
$$\wh W_k = \begin{pmatrix} \wh V_{2k} \\ \wh V_{2k+1} \end{pmatrix} \hspace{20pt} \text{and} \hspace{20pt} J = \begin{pmatrix} 0 & 1 \\ 1 & 0 \end{pmatrix}.$$
It follows from \eqref{estnu} that
$$Q_{n}(\wh \nu_n -\nu_n) = \sum_{k\in\T_{n-1}}(\wh V_{2k} - V_{2k})(\wh V_{2k+1} - V_{2k+1}) \psi_k + R_n.$$
Furthermore, one can observe that $(R_n)$ is a square integrable martingale with increasing process
\begin{align*}
\langle R \rangle _n &= \sum_{l=0}^{n-1} \sum_{k\in\G_l} \E[(\wh W_k - W_k)^t J W_k W_k^t J (\wh W_k - W_k)\psi_k\psi_k^t|\calF_l]\as\\
	&= \sum_{l=0}^{n-1} \sum_{k\in\G_l} (\wh W_k - W_k)^t J \E[W_k W_k^t
|\calF_l] J (\wh W_k - W_k)\psi_k\psi_k^t\as\\
	&= \sum_{l=0}^{n-1} \sum_{k\in\G_l} (\wh W_k - W_k)^t J \begin{pmatrix} P(X_k) & Q(X_k) \\ Q(X_k) & R(X_k) \end{pmatrix} J (\wh W_k - W_k)\psi_k\psi_k^t\as\\
&= \sum_{l=0}^{n-1} \sum_{k\in\G_l} (\wh W_k - W_k)^t \begin{pmatrix} R(X_k) & Q(X_k) \\ Q(X_k) & P(X_k) \end{pmatrix} (\wh W_k - W_k)\psi_k\psi_k^t\as\\
\end{align*}
Then, as previously, Lemma \ref{CVsomme} and \eqref{quadratic1} lead to $\|\langle R \rangle_n\| = \calO(n)$ a.s. which allows us to say that $R_n = o(n)$ a.s. Furthermore
\begin{align*}
\left\| \sum_{k\in\T_{n-1}}(\wh V_{2k} - V_{2k})(\wh V_{2k+1} - V_{2k+1}) \psi_k \right\|&\\
&\hspace{-30pt} \leq\frac12 \sum_{k\in\T_{n-1}} \left((\wh V_{2k} - V_{2k})^2 + (\wh V_{2k+1} - V_{2k+1})^2\right) \|\psi_k\|,\\
&\hspace{-30pt} \leq \frac12 \sum_{l=0}^{n-1} \|\wh \theta_n -\theta\|^2 \sum_{k\in\G_l} \|\Phi_k\|^2 \|\psi_k\|,
\end{align*}
which implies, thanks to Lemma \ref{CVsomme} and \eqref{quadratic1}, that
$$\left\| \sum_{k\in\T_{n-1}}(\wh V_{2k} - V_{2k})(\wh V_{2k+1} - V_{2k+1}) \psi_k \right\| = \calO(n)\as$$
Finally, we infer from \eqref{limQ} that
$$\|\wh \nu_n - \nu_n\| = \calO\left(\frac{n}{|\T_{n-1}|}\right)\as$$
It remains to prove the a.s. convergence of $\eta_n$, $\zeta_n$ and $\nu_n$ to $\eta$, $\zeta$ and $\nu$, respectively which would immediately imply the a.s. convergence of our estimates through \eqref{vitesseeta}, \eqref{vitesseetad} and \eqref{vitesserho}. Denote
\begin{equation}
N_n = Q_{n-1}(\eta_n-\eta)=\sum_{k\in\T_{n-1}} v_{2k} \psi_k \label{defN}
\end{equation}
where $v_{2n}=V_{2n}^2-\eta^t\psi_n$. One can observe that $(N_n)$ is a square integrable martingale with increasing process $\langle N \rangle_n$ given by
\begin{align*}
\langle N \rangle_n &= \sum_{l=0}^{n-1} \sum_{k\in\G_l} \E[v_{2k}^2|\calF_l] \psi_k \psi_k^t \as\\
\end{align*}
Hence, if $\gamma = \max(\mu_a^4-\sigma_a^4,2\sigma_a^2\sigma_c^2,\mu_c^4-\sigma_c^4)$, we obtain that
\begin{align*}
\|\langle N \rangle_n\| &\leq \left\|\sum_{l=0}^{n-1} \sum_{k\in\G_l} \gamma (1+X_k^2)^2 \psi_k \psi_k^t\right\|\as\\
	&\leq \gamma \sum_{k\in\T_{n-1}} (1+X_k^2)^2 \|\psi_k\|^2 = \gamma \sum_{k\in\T_{n-1}} (1+X_k^2)^4\as
\end{align*}
which leads, via Lemma \ref{CVsomme}, to
$$\|\langle N \rangle_n\| = \calO(|\T_{n-1}|)\as$$
Consequently,
$$\|N_n\|^2 = \calO(n|\T_{n-1}|) \as$$
Then, we deduce from \eqref{limQ} and \eqref{defN} that $\eta_n$ converges a.s.~to $\eta$ with the a.s.~rate of convergence given in Remark \ref{remrate}. The proof concerning the a.s. convergence of $\zeta_n$ to $\zeta$ and the second rate of convergence in Remark \ref{remrate} is exactly the same. Hereafter, denote
\begin{equation}
H_n = Q_{n-1}(\nu_n - \nu) = \sum_{k\in\T_{n-1}} w_{2k} \psi_k \label{defH}
\end{equation}
where $w_{2n} = V_{2n}V_{2n+1} - \nu^t\psi_n$. Once again, the sequence $(H_n)$ is a square integrable martingale with increasing process
\begin{align*}
\langle H \rangle_n = \sum_{l=0}^{n-1} \sum_{k\in\G_l} \E[w_{2k}^2|\calF_l] \psi_k \psi_k^t \as 
\end{align*}
Moreover, if $\alpha= \max(\nu_{ab}^2,\nu_{cd}^2,(\sigma_a^2+\sigma_c^2)(\sigma_b^2+\sigma_d^2))$, we find that
\begin{align*}
\|\langle H \rangle_n\| &\leq \left\|\sum_{l=0}^{n-1} \sum_{k\in\G_l} \alpha (1+X_k^2)^2 \psi_k \psi_k^t\right\|\as\\
	&\leq \alpha \sum_{k\in\T_{n-1}} (1+X_k^2)^2 \|\psi_k\|^2 = \alpha \sum_{k\in\T_{n-1}} (1+X_k^2)^4\as
\end{align*}
which allows us to say, as previously, that
$$\|H_n\|^2 = \calO(n|\T_{n-1}|) \hspace{20pt} \text{and} \hspace{20pt} \|\nu_n - \nu\|^2 = \calO\left(\frac{n}{|\T_{n-1}|}\right)\as$$
It clearly proves the a.s. convergence of $\nu_n$ to $\nu$ with the last a.s.~rate of convergence given in Remark \ref{remrate}, which completes the proof of Theorem \ref{CVpsvar}.


\section{Proofs of the asymptotic normalities}


The key point of the proof of the asymptotic normality of our estimators is the central limit theorem for triangular array of vector martingale given e.g. in Theorem 2.1.9 of \cite{Duflo}. With this aim in mind, we will change the filtration considering, instead of the generation wise filtration $(\calF_n)$, the sister-pair wise filtration $(\calG_n)$ given by
$$\calG_n = \sigma\left\{X_1,(X_{2k},X_{2k+1}),1\leq k\leq n\right\}.$$
\subsection{Proof of convergence \eqref{TCLtheta}}
We will consider the triangular array of vector martingale $(M_k^{(n)})$ defined as
\begin{equation}\label{defMnk}
M_k^{(n)} = \frac1{\sqrt{|\T_n|}} \sum_{l=1}^k D_l,
\end{equation}
where 
$$D_l = \begin{pmatrix} X_lV_{2l} \\ V_{2l} \\ X_lV_{2l+1} \\V_{2l+1} \end{pmatrix}.$$
It is obvious that $\left(M^{(n)}\right)$ is a square integrable vector valued martingale with respect to the filtration $(\calG_k)$. Moreover, we can observe that
\begin{equation}\label{lienM}
M_{t_n}^{(n)} = \frac1{\sqrt{|\T_n|}} \sum_{l=1}^{t_n} D_l = \frac1{\sqrt{|\T_n|}} M_{n+1}
\end{equation}
where $t_n=|\T_n|=2^{n+1}-1$. In addition, the increasing process of this square integrable martingale is given by
\begin{align*}
\langle M^{(n)}\rangle_k &= \frac1{|\T_n|} \sum_{l=1}^{k} \E[D_l D_l^t|\calG_{l-1}],\\
	&= \frac1{|\T_n|} \sum_{l=1}^k \begin{pmatrix}P(X_l) & Q(X_l) \\ Q(X_l) & R(X_l) \end{pmatrix} \otimes \begin{pmatrix}X_l^2 & X_l \\ X_l & 1\end{pmatrix} \as
\end{align*}
Then, \eqref{limcrochet} leads to
$$\lim_{n\to\infty} \langle M^{(n)}\rangle _{t_n} = L \as$$
We will now establish Lindeberg's condition thanks to Lyapunov's condition. Let
$$\phi_n=\sum_{k=1}^{t_n} \E\left[\left.\|M_k^{(n)} - M_{k-1}^{(n)}\|^4\right|\calG_{k-1}\right].$$
It follows from \eqref{defMnk} that 
\begin{align*}
\phi_n &= \frac1{|\T_n|^2}\sum_{k=1}^{t_n} \E\left[\left.(1+X_k^2)^2(V_{2k}^2+V_{2k+1}^2)^2\right|\calG_{k-1}\right],\\
	&\leq \frac2{|\T_n|^2}\sum_{k=1}^{t_n} \E\left[\left.(1+X_k^2)^2(V_{2k}^4+V_{2k+1}^4)\right|\calG_{k-1}\right].
\end{align*}
Since we already saw in Section \ref{preuveCVpstheta} that
$$\E\left[V_{2k}^4|\calF_n\right] \leq \mu_{ac}^4(1+X_k^2)^2  \hspace{20pt} \text{ and } \hspace{20pt} \E[V_{2k+1}^4|\calF_n] \leq \mu_{bc}^4(1+X_k^2)^2 \as$$
where $\mu_{ac}^4=\max(\mu_a^4,3\sigma_a^2\sigma_c^2,\mu_c^4)$ and $\mu_{bd}^4=\max(\mu_b^4,3\sigma_b^2\sigma_d^2,\mu_d^4)$, we have that
$$\phi_n \leq \frac{2(\mu_{ac}^4+\mu_{bd}^4)}{|\T_n|^2} \sum_{k=1}^{t_n}(1+X_k^2)^4\as$$
leading, via Lemma \ref{CVsomme}, to
$$\lim_{n\to\infty} \phi_n = 0\as$$
Consequently, Lyapunov's condition is satisfied and Theorem 2.1.9 of \cite{Duflo} together with \eqref{lienM} imply that
$$\frac{1}{\sqrt{|\T_{n-1}|}} M_n \liml \calN(0,L).$$
Moreover, we easily obtain from Lemma \ref{CVsomme} that
\begin{equation}\lim_{n\to\infty} \frac{\Sigma_n}{|\T_n|} = I_2\otimes C= \Gamma \as\label{limSigma}\end{equation}
where $C$ is the positive definite matrix given by \eqref{defCD}. Finally, we deduce from \eqref{diffthetaM} together with \eqref{limSigma} and Slutsky's lemma that
$$\sqrt{|\T_{n-1}|}(\wh \theta_n - \theta) \liml \calN(0,\Gamma^{-1} L \Gamma^{-1}).$$
\subsection{Proof of convergence \eqref{TCLeta}}
We will now consider the triangular array of vector martingale $(N_k^{(n)})$ defined as
\begin{equation*}
N_k^{(n)} = \frac1{\sqrt{|\T_n|}} \sum_{l=1}^k v_{2l}\psi_l.
\end{equation*}
It is obvious from \eqref{defN} that
\begin{equation}\label{lienN}
N_{t_n}^{(n)} = \frac1{\sqrt{|\T_n|}}Q_n(\eta_{n+1}-\eta)=\frac1{\sqrt{|\T_n|}}N_{n+1}.
\end{equation}
Moreover, we also have
$$\E[v_{2n}^2|\calG_{n-1}] = (\mu_a^4-\sigma_a^4)X_n^4+4\sigma_a^2\sigma_c^2X_n^2+(\mu_c^4-\sigma_c^4).$$
Hence, the increasing process associated to the square integrable martingale $\left(N^{(n)}\right)$ is given by
\begin{align*}
\langle N^{(n)}\rangle_k &= \frac1{|\T_n|} \sum_{l=1}^k \E\left[v_{2l}^2\psi_l\psi_l^t |\calG_{l-1}\right],\\
&=\frac1{|\T_n|}\sum_{l=1}^k\left((\mu_a^4-\sigma_a^4)X_l^4+4\sigma_a^2\sigma_c^2X_l^2+(\mu_c^4-\sigma_c^4)\right)\psi_l\psi_l^t,
\end{align*}
and Lemma \ref{CVsomme} allows us to say that
$$\lim_{n\to\infty} \langle N^{(n)}\rangle_{t_n} = M_{ac} \as$$
As previously, we now need to check Lyapunov's condition. For $\alpha>4$ such that \H{H5} is satisfied, denote
$$\phi_n = \sum_{k=1}^{t_n} \E\left[\left.\|N_k^{(n)} - N_{k-1}^{(n)}\|^{\alpha/2}\right|\calG_{k-1} \right].$$
We clearly have
$$\|N_k^{(n)} - N_{k-1}^{(n)}\|^2 = \frac1{|\T_n|}v_{2k}^2\|\psi_k\|^2 \leq  \frac1{|\T_n|}v_{2k}^2 (1+X_k^2)^2,$$
leading to 
\begin{equation}\|N_k^{(n)} - N_{k-1}^{(n)}\|^{\alpha/2} \leq \frac1{|\T_n|^{\alpha/4}}|v_{2k}|^{\alpha/2} (1+X_k^2)^{\alpha/2}.\label{majN}\end{equation}
Moreover, it exists some constant $\beta>0$ such that
\begin{align}
|v_{2k}|^{\alpha/2} &\leq (V_{2k}^2+\sigma_a^2X_k^2+\sigma_c^2)^{\alpha/2},\nonumber\\
	&\leq \beta(|V_{2k}|^{\alpha}+|X_k|^{\alpha}+1).\label{majV6}
\end{align}
In addition, it exists some constant $\gamma>0$ such that
\begin{align}
|V_{2k}|^\alpha&\leq\(|a_k-a||X_k|+|\veps_{2k}-c|\)^{\alpha},\nonumber\\
&\leq\gamma\(|a_k-a|^{\alpha}|X_k|^{\alpha}+|\veps_{2k}-c|^{\alpha}\).\label{majV2k5}
\end{align}
Denote
$$Y=\max\(\sup_{n\geq0}\sup_{k\in\G_n} \E[|a_k-a|^\alpha|\calF_n],\sup_{n\geq0}\sup_{k\in\G_n} \E[|\veps_{2k}-c|^\alpha|\calF_n]\).$$
It clearly follows from \eqref{majV2k5} that
$$\E[|V_{2k}|^\alpha|\calG_{k-1}] \leq \gamma Y(|X_k|^\alpha+1)\as$$
Consequently, we deduce from \eqref{majN} and \eqref{majV6} that it exists some constants $\delta>0$ and $\zeta>0$ such that
\begin{align*}
\phi_n&\leq\sum_{k=1}^{t_n} \frac{\beta(1+\gamma Y)}{|\T_n|^{\alpha/4}}(1+|X_k|^\alpha)(1+X_k^2)^{\alpha/2}\as\\
	&\leq\sum_{k=1}^{t_n} \frac{\delta(1+Y)}{|\T_n|^{\alpha/4}}(1+X_k^2)^\alpha\as\\
	&\leq\frac{\zeta(1+Y)}{|\T_n|^{\xi}}\frac{1}{|\T_n|}\sum_{k=1}^{t_n} (1+X_k^{2\alpha})\as
\end{align*}
where $\xi=\alpha/4-1>0$. Moreover, we can obviously suppose that $\alpha\leq5$ and we can prove via the same lines as in Section \ref{preuvepremlem} that, as soon as $\E[a_n^{10}]<1$, $\E[b_n^{10}]<1$ and
$$\sup_{n\geq1}\E[\veps_{2n}^{10}]<\infty \hspace{20pt} \text{and} \hspace{20pt} \sup_{n\geq1}\E[\veps_{2n+1}^{10}]<\infty$$
it exists some positive constant $\mu$ such that, for all $n\geq0$, $\E[X_n^{10}]<\mu$. Therefore, the Borel-Cantelli lemma clearly ensures that
$$\lim_{n\to\infty} \frac1{|\T_n|^\xi}\frac1{|\T_n}\sum_{k=1}^{t_n}(1+X_k^{10})=0\as$$
which implies that
$$\lim_{n\to\infty} \phi_n = 0 \as$$
Thus, Lyapunov's condition is satisfied and we infer from Theorem 2.1.9 of \cite{Duflo} and \eqref{lienN} that
\begin{equation}\label{limlN}
\frac1{\sqrt{|\T_{n-1}|}}N_n \liml \calN(0,M_{ac}).
\end{equation}
Finally, \eqref{limQ} together with \eqref{limlN} and Slutsky's lemma allow us to say that
$$\sqrt{|\T_{n-1}|}(\eta_n - \eta) \liml \calN(0,Q^{-1}M_{ac}Q^{-1})$$
implying, through \eqref{vitesseeta}, that
$$\sqrt{|\T_{n-1}|}(\wh \eta_n - \eta) \liml \calN(0,Q^{-1}M_{ac}Q^{-1}).$$
The proof of \eqref{TCLetad} follows exactly the same lines.\\
\subsection{Proof of convergence \eqref{TCLrho}}
The last step is to prove the asymptotic normality given by \eqref{TCLrho}. We will once again consider a triangular array of vector martingales $(H^{(n)}_k)$ defined as
$$H_k^{(n)} = \frac1{\sqrt{|\T_n|}} \sum_{l=1}^{k} w_{2l} \psi_l.$$
We clearly have
$$H_{t_n}^{(n)} = \frac1{\sqrt{|\T_n|}} \sum_{l=1}^{t_n} w_{2l} \psi_l = \frac1{\sqrt{|\T_n|}} H_n.$$
Consequently, $\left(H^{(n)}\right)$ is a square integrable martingale with increasing process given by
$$\langle H^{(n)}\rangle_k = \frac1{|\T_n|}\sum_{l=1}^k\E[w_{2l}^2|\calG_{l-1}]\psi_l\psi_l^t.$$
Moreover, we can easily obtain that
$$\E[w_{2n}^2|\calG_{n-1}] = (\nu_{ab}^2-\rho_{ab}^2)X_n^4+(\sigma_a^2\sigma_d^2+\sigma_b^2\sigma_c^2+2\rho_{ab}\rho_{cd})X_n^2+(\nu_{cd}^2-\rho_{cd}^2) \as$$
Therefore, it follows from Lemma \ref{CVsomme} that
$$\lim_{n\to\infty} \langle H^{(n)}\rangle_{t_n} = H \as$$
Let us now verify that Lyapunov's condition is satisfied. For $\alpha>4$ such that \H{H5} is verified, denote
$$\phi_n = \sum_{i=1}^{t_n}\E\left[\left.\|H^{(n)}_k-H_{k-1}^{(n)}\|^{\alpha/2}\right|\calG_{k-1}\right].$$
As previously, we obtain that
$$\|H^{(n)}_k-H_{k-1}^{(n)}\|^{\alpha/2} \leq \frac1{|\T_{n}|^{\alpha/4}} |w_{2k}|^{\alpha/2}(1+X_k^2)^{\alpha/2},$$
and we can see that
\begin{align*}
|w_{2k}|&\leq|V_{2k}V_{2k+1}|+|\rho_{ab}|X_k^2+|\rho_{cd}|,\\
&\leq \frac12V_{2k}^2+\frac12V_{2k+1}^2+|\rho_{ab}|X_k^2+|\rho_{cd}|.
\end{align*}
We deduce from the previous calculations that it exists some constant $\xi>0$ and some a.s.~finite random variable $Y$ such that
$$\E[|w_{2k}|^{\alpha/2}|\calG_{k-1}]\leq\xi(1+Y)(1+|X_k|^\alpha)\as$$
It leads, for some constant $\zeta>0$, to
$$\E[\|H^{(n)}_k-H_{k-1}^{(n)}\|^{\alpha/2}|\calG_{k-1}]\leq \frac{\zeta(1+Y)}{|\T_n|^{\alpha/4}} (1+X_k^{2\alpha})\as$$
Therefore, as before, we find that
$$\lim_{n\to\infty}\phi_n=0\as$$
Finally, we obtain that
$$\frac1{\sqrt{|\T_{n-1}|}}H_n\liml\calN(0,H),$$
alternatively
$$\sqrt{|\T_{n-1}}(\nu_n-\nu)\liml\calN(0,Q^{-1}HQ^{-1}),$$
which, via \eqref{vitesserho} allows us to conclude that
$$\sqrt{|\T_{n-1}}(\wh\nu_n-\nu)\liml\calN(0,Q^{-1}HQ^{-1}).$$

\nocite{*}
\bibliographystyle{acm}
\bibliography{RCBARnotweighted}

\end{document}